\documentclass[12pt]{amsart}

\usepackage{amsmath}
\usepackage{graphicx}
\usepackage{enumerate}
\usepackage{url} 
\usepackage{amsmath,amsfonts,amssymb,amsthm,booktabs,color,epsfig,graphicx,hyperref,url}
\usepackage{float}
\usepackage{grffile}
\usepackage{apptools}
\usepackage{subcaption}
\usepackage{adjustbox}
\usepackage{fancyvrb}
\usepackage{txfonts}
\usepackage{tabularx}
\usepackage{adjustbox}
\usepackage{pdfpages} 
\usepackage{pgffor} 

\makeatletter
\AtBeginDocument{\let\LS@rot\@undefined}
\makeatother

\def\supplementfilename{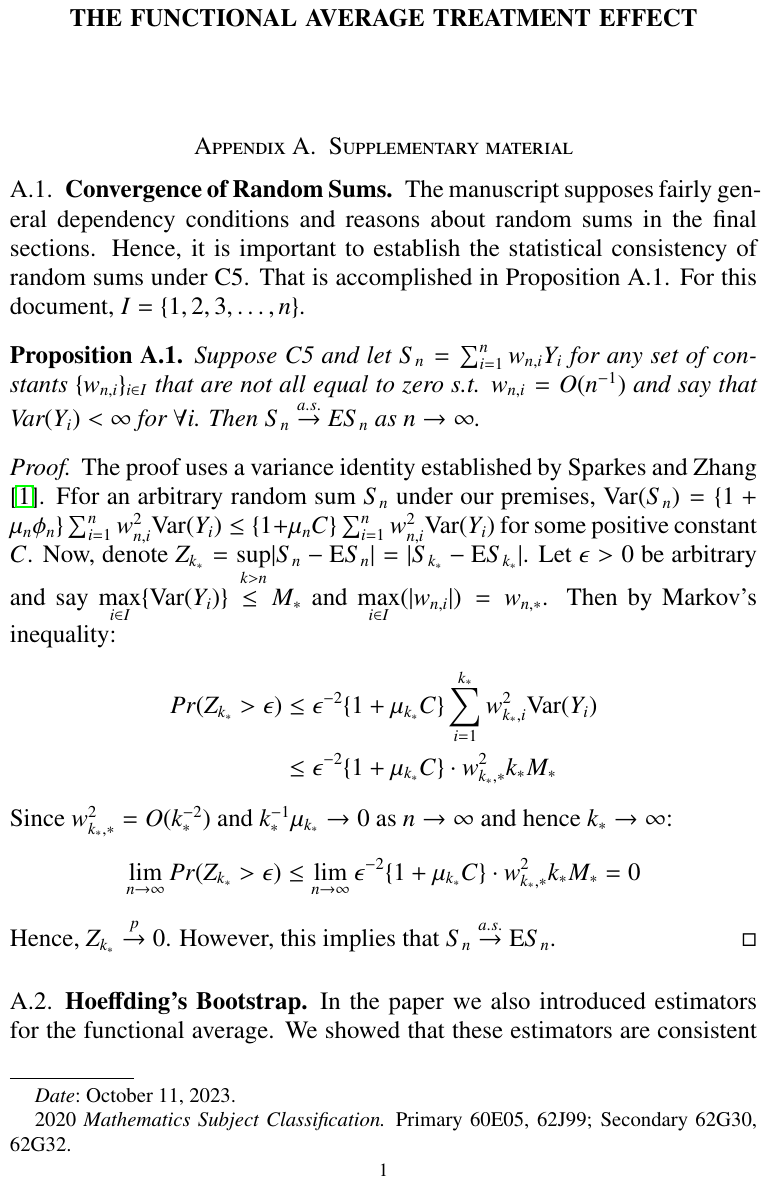}

\pdfximage{\supplementfilename}
\def\numbersupplementpages{\the\pdflastximagepages}

\newif\ifarXiv
\arXivtrue

\theoremstyle{plain}

\newtheorem{proposition}{Proposition}

\theoremstyle{definition}
\newtheorem{definition}{Definition}

\theoremstyle{remark}

\numberwithin{equation}{section}

\begin{document}

\title{The Functional Average Treatment Effect}

\author{Shane Sparkes}
\address{Department of Population and Public Health Sciences, University of Southern California, Los Angeles, California 90033}
\curraddr{Department of Population and Public Health Sciences, University of Southern California, Los Angeles, California 90033}
\email{sgugliel@usc.edu}

\author{Erika Garcia}
\address{Department of Population and Public Health Sciences, University of Southern California, Los Angeles, California 90033}
\curraddr{Department of Population and Public Health Sciences, University of Southern California, Los Angeles, California 90033}
\email{garc991@usc.edu}

\author{Lu Zhang}
\address{Department of Population and Public Health Sciences, University of Southern California, Los Angeles, California 90033}
\curraddr{Department of Population and Public Health Sciences, University of Southern California, Los Angeles, California 90033}
\email{lzhang63@usc.edu}


\date{October 11, 2023}



\begin{abstract}
This paper establishes the functional average as an important estimand for causal inference. The significance of the estimand lies in its robustness against traditional issues of confounding. We prove that this robustness holds even when the probability distribution of the outcome, conditional on treatment or some other vector of adjusting variables, differs almost arbitrarily from its counterfactual analogue. This paper also examines possible estimators of the functional average, including the sample mid-range, and proposes a new type of bootstrap for robust statistical inference: the Hoeffding bootstrap. After this, the paper explores a new class of variables, the $\mathcal{U}$ class of variables, that simplifies the estimation of functional averages. This class of variables is also used to establish mean exchangeability in some cases and to provide the results of elementary statistical procedures, such as linear regression and the analysis of variance, with causal interpretations. Simulation evidence is provided. The methods of this paper are also applied to a National Health and Nutrition Survey data set to investigate the causal effect of exercise on the blood pressure of adult smokers.
\end{abstract}

\maketitle
\section{Introduction}
The Neyman-Rubin (NR) model is an important framework for causal inference with observational designs \cite{rubin2019essential, pearl2010causal, holland1986statistics, imbens2015causal}. Say we are interested in studying a population of random variables $\{ Y_i \}_{i \in I}$, where $I = \{1, \ldots, N \}$, conditional on exposure to $T_i$. Here, we specify $T_i$ as binary for simplicity and denote this conditional variable as $Y_{t_i}$. At any point in time, it is impossible to observe both $Y_{t_i=1}$ and $Y_{t_i=0}$ for an arbitrary unit $i$: a fact that makes individual causal contrasts, such as $Y_{t_i=1} - Y_{t_i=0}$, undefined. This problem characterizes the `fundamental problem of causal inference' \cite{ding2018causal}. The NR model frames this challenge in the language of missing outcomes and constructs populations of counterfactual probability distributions to address it, say $\{ Y^{t=1}_i \}_{i \in I}$ and $\{ Y^{t=0}_i \}_{i \in I}$. These populations represent a hypothetical situation such that (s.t.) the treatment status of the entire population has been experimentally fixed to $T=t$. The goal under the NR framework is to identify \textit{summary} causal effects that would have occurred, provided one could actually have constructed and manipulated both $\{ Y^{t=1}_i \}_{i \in I}$ and $\{ Y^{t=0}_i \}_{i \in I}$ \cite{hernan2010causal}. More specifically, under this rubric, we wish to identify some multivariate function $g$ and some vector of adjusting variables $\mathbf{L}$ s.t. $\text{E} \{ g(Y_{t_1=1}, \ldots, Y_{t_{n_1} =1}) - g(Y_{t_1=0}, \ldots, Y_{t_{n_0} =0}) | \mathbf{L} \}$ equals its counterpart evaluated from the elements of the hypothetical populations, at least asymptotically \cite{imbens2010rubin}. Here, we again use the difference of estimands only as an example. When this is possible, unmeasured variables are said to be ignorable, conditional on $\mathbf{L}$, and the conclusions of the observational study are equivalent to those of an experimental one \cite{rosenbaum1983central, holland1987causal}.

Although $g$ can be any function that has scientifically meaningful properties, a small number of summary functions have dominated the literature. Estimands related to quantiles have commandeered some attention, for instance \cite{jin2008principal, imbens2010rubin, belloni2017program, gangl2010causal}. A lion's share, however, has been claimed by the arithmetic mean. Unfortunately, the adage that `there is no free lunch' applies to this function. When a researcher wishes to replace $\text{E}Y^{t}$ with $\text{E}Y_{t}$ to estimate the \textit{expected} treatment effect, a set of additional conditions are required: consistency (C1), mean exchangeability (C2), and positivity (C3) \cite{cole2009consistency}. These conditions will soon be defined in detail. For now, it suffices to state that researchers often attempt to achieve C2 conditional on a vector of adjusting variables $\mathbf{L}$ so that $\text{E}Y_{t, \mathbf{L}} = \text{E}Y_{\mathbf{L}}^{t}$. Standardization (iterated expectation) is then used to yield the quantity of interest since $\text{E}\{ \text{E} Y_{\mathbf{L}}^t\} = \text{E}Y^t$ under C1-C3.  A cardinal problem is that C2 is non-trivial to achieve in observational studies, even conditional on some random vector $\mathbf{L}$ \cite{hernan2006estimating, greenland1999causal, greenland1999confounding}. Moreover, $\mathbf{L}$ is often high in dimension. Parametric methods must then be employed to approximate the mean model of $Y$ in conjunction with standardization and bias likely ensues outside of toy examples \cite{hernan2010causal}.

One of our cardinal contributions is to highlight a different summary causal effect---the functional average---as a valuable estimand for the NR framework since it avoids many of these challenges, it can impart a causal interpretation to the results of standard statistical procedures, and it is identifiable under mild conditions. For example, if $Y_t$ and $Y^t$ have the same image in the traditional analysis sense, this is sufficient. Their probability distributions can otherwise differ arbitrarily. Confounding is immaterial, insofar as it does not change the image of the underlying function(s). So is informative sampling more generally. All that matters for identification is that---theoretically---$Y_t$ and $Y^t$ pull from the same set of real numbers, at least conditional on some $\mathbf{L}$. Although not necessary, this is at least sufficient for establishing what we call \textit{functional average} exchangeability.

The remainder of this paper goes as follows. For clarity, we mathematically define functional averages and prove an elementary but fundamental claim in Section 2. Moreover, we show that the functional average treatment effect is salient when the researcher believes that an intervention alters the set of possible values that an outcome can achieve or no expected causal effect exists. After this, we examine a small set of functional average estimators, including the sample mid-range, and establish their statistical consistency under general conditions. Since their sampling distributions are largely intractable, however, we re-purpose the bootstrap as a method for conservative inference. In Section 3, we provide elucidation on a particular class of bounded random variables---the $\mathcal{U}$ class---that generalizes the notion of symmetry and assists in the estimation of functional averages. We also show that $\mathcal{U}$ random variables possess many favorable properties when it comes to causal inference. These facts allow us to also prove---under the auspices as a causal theory---that linear regressions estimate causal effects under a standard set of assumptions already employed for associational studies. Section 4 presents simulation evidence that substantiates our claims.

Finally, in Section 5, we use our strategies in conjunction with data from the National Health and Nutrition Examination Survey Data I Epidemiologic Follow-up Study (NHEFS) to investigate if exercise activity causally impacts the systolic blood pressure (SBP) of adult smokers. Plentiful evidence exists that smoking is associated with cardiovascular disease processes and mortality \cite{glantz1991passive, stallones2015association}. Evidence has also been presented that smoking is a factor in arterial stiffening \cite{narkiewicz2005smoking}. However, while some literature has supported the proposition that exercise lowers arterial blood pressure \cite{elley2002aerobic} and that smokers who exercise show fewer signs of arterial stiffening \cite{park2014does}, the evidence is not yet definitive. Functional average estimation targets deterministic changes in the structure of an outcome variable and is thus an informative tool in this context.

\section{Functional Average Treatment Effect}

In this section, we first introduce important definitions and notation, although some concepts will be left implicit for readability. For instance, we leave the underlying probability space of the form $(\Omega, \mathcal{F}, \mathcal{P})$ for an arbitrary random variable $Y(\omega): \Omega \to \mathbb{R}$ unstated, and the same goes for probability spaces defining joint distributions. Recall that the support of a random variable is a smallest closed set $\mathcal{S}$ s.t. $\text{Pr}(Y \in \mathcal{S}) = 1$. Alternatively, it can also be defined as the closure of the set of values $\mathcal{S}$ s.t. the density or mass function $f(y) > 0$ for $\forall y \in \mathcal{S}$. Here, we will be dealing with bounded random variables, which means that $\mathcal{S}$ is a strict subset of the real numbers. This is not a limiting constraint. Anything that can be empirically measured is necessarily bounded.

With these concepts, we can revisit the functional average. If $\mathcal{S}$ is discrete, define $R = |\mathcal{S}|$, where $|\cdot|$ in this context denotes the number of elements in the set. If $Y$ is continuous, then $R = \int_{\mathbb{R}} \{ 1_{y \in \mathcal{S}} \} dy$ and the functional average $\text{Av}(\cdot)$ is $\text{Av}(Y) = R^{-1} \int_{\mathbb{R}} \{ y 1_{y \in \mathcal{S}} \} dy$. For discrete variables, it is $\text{Av}(Y) = R^{-1} \sum_{y_i \in \mathcal{S}} y_i$. Note that we have avoided the use of general measures for purposes of accessibility.

Sometimes it will be the case that, for some measurable function $g$, $Y = g(X_1, \ldots, X_k)$. Then the support of $Y$ with respect to (w.r.t.) the joint distribution of $(X_1, \ldots, X_k) = \mathbf{X} \in \mathbb{R}^k$ is some general region $\mathcal{R} \subseteq \mathcal{S}_1 \times \cdots \times \mathcal{S}_k$, where each $\mathcal{S}_i$ indicates the support of $X_i$. Without loss of generality (WLOG), we will henceforth deal only with the continuous case. In this context, $R = \int_{\mathbb{R}^k} 1_{(x_1, \ldots, x_k) \in \mathcal{R}} dx_1 \cdots dx_k$ and $\text{Av}_{\mathbf{x}}(Y) = R^{-1} \int_{\mathcal{R}} g(x_1, \ldots, x_k) dx_1 \cdots dx_k$.

Now, let $\text{E}_{h}Y$ indicate that the expectation of $Y$ is taken w.r.t. a different density or mass function $h(y)$ that is also defined on $\mathcal{S}$. Then it is also apparent that $\text{Av}(Y) = \text{E}_hY$ when $h(y)$ is a uniform density or mass function. When a subscript is omitted and $Y \sim f(y)$, it will be understood that the expectation is taken w.r.t. the baseline density (mass) function $f$, provided it exists. Otherwise, we say that $\text{E}_U Y = \text{Av}(Y)$ as a special case, although we will mostly avoid this notation. This is because $\text{Av}(Y)$ is best interpreted with the lenses of basic, deterministic analysis. The exception to this statement is when $Y$ truly follows a uniform probability law.

A functional average treatment effect then---for any two treatment values of interest $t, t'$---can be defined as $h\{ \text{Av}(Y^t), \text{Av}(Y^{t'}) \} = h\{ \text{E}_U Y^t, \text{E}_UY^{t'} \}$ for a user-specified function $h$. In this paper, we set $h$ to a simple difference for exploratory purposes, i.e., $h\{ \text{Av}(Y^t), \text{Av}(Y^{t'}) \}= \Delta_{t, t'} =  \text{Av}(Y^t) - \text{Av}(Y^{t'})$.

\subsection{Examples of Applicability}

The average functional value is not a usual focus in statistical settings. The expected value w.r.t. the baseline measure has instead largely been an object of interest. Hence, we offer a short argument and demonstration of its importance as a preliminary apologia. By definition, an expected value is a sum of all possible values, where each value is weighted by the probability of observation or its density. However, the chance (or density) of observation is extraneous to causal relationships that are unrelated to altered probabilities. Although the functional average can also be construed, albeit counterfactually in most cases, as an expected value, the uniform measure imbues it with a more deterministic interpretation. This is because it does not require a probabilistic framework, although such a framework is often necessary for its estimation. 

\subsubsection{Example 1} Say $T$ is a binary treatment variable s.t. $T=1$ when a particular psychotropic medication is received and $Y$ is a Likert scale measuring anxiety in individuals with clinical depression. Also say that $Y_{t=0}$ can take any integer between one and ten with the following probabilities:
$$\{.01, .04, .05, .1, .15, .15, .3, .1, .05, .05 \}$$

Then $\text{E}Y_{t=0} = 6.14$ and $\text{Av}(Y_{t=0}) = 5.5$. Now, say $Y_{t=1}$ has non-zero mass only on integers between one and eight with the following probabilities: $\{.01, .01, .01, .05, .1, .5, .18, .14\}$. Under this scheme, it is also the case that $\text{E}Y_{t=1} = 6.14$. This makes the detection of a causal effect impossible if only the expected treatment effect is utilized. However, $\text{Av}(Y_{t=1}) = 4.5$, a value that possibly reflects the elimination of extreme anxiety under treatment, albeit at the cost of more mild to moderate anxiety experiences.

\subsubsection{Example 2} Again, let $T$ be a binary treatment variable for simplicity and say $Y_{t=0}$ is a random variable for the untreated systolic blood pressure (SBP) of individuals who have been diagnosed with high blood pressure. We will assume that $Y_{t=0}$ follows a truncated normal distribution with a mean at 155 mmHg and support $\mathcal{S}_{t=0} = [110, 370]$ and $Y_{t=1}$ follows a truncated normal distribution with mean 125 mmHg on support $\mathcal{S}_{t=1} = [90, 250]$. Here, both an expected and a functional average treatment effect are possibly present and relevant. For the latter, $\text{Av}(Y_{t=0}) = 240$, while $\text{Av}(Y_{t=1}) = 170$. The functional average in this example offers additional information pertaining to changes in the possibilities of extremes that cannot be captured by expected values with the baseline probability measure, insofar as it is not uniform. 

Moreover, these examples elucidate how functional averages and their differences remain invariant to any redistribution of the presented probabilities insofar as they remain non-zero. For instance, say we employed a biased sampling mechanism (such as a convenience sampling) and we also failed to measure all confounders for Example 2. As a consequence, say $Y_{t=1}$ follows a truncated normal distribution s.t. $\text{E}Y_{t=1} = 110$ and $Y_{t=0}$ follows a truncated normal distribution s.t. $\text{E}Y_{t=0} = 180$. These facts would not matter insofar as the convenience sample was executed in such a way as to preserve the sets of values that the functions could theoretically materialize under an experimental design. For some estimators, preservation of the extremes alone is sufficient.

\subsection{Identifying and Estimating Counterfactual Functional Averages}

Next, we prove some basic statements about functional averages under mild conditions and the rubric of informative sampling. For this, we specify a conditional population of interest $P= \{ Y_{t, 1}, \ldots , Y_{t, N} \}$ s.t. $Y_{t, 1} \sim f(y| T=t)$ WLOG. This setup can be defined with additional conditioning or extended to unconditional circumstances, but this is omitted here for brevity. Additionally, observe a complete-case sample $\zeta \subset P$ and a complementary vector of indicator variables $\boldsymbol{\delta} = (\delta_1, \ldots, \delta_N)$ s.t. $\delta_i = 1$ if and only if $Y_{t, i} \in \zeta$. We also assume that $\text{E}(\delta_i| y_i, t) > 0$ for all $\forall i$, which implies that $E(\delta_i | t) =\pi_i > 0$ for all $i$, an assumption that is typically called sampling positivity. It is well-known that an arbitrary $Y_{t, i} \in \zeta$ does not, in general, follow the distribution of the theoretical population \cite{pfeffermann2009inference, pfeffermann1998parametric, patil1987weighted, patil1978weighted}. Instead, $Y_{t, i} \in \zeta$ possesses a weighted density or mass function $f_{\delta}(y_{i}| t) = \pi^{-1}_i \text{E}(\delta_i | y_{i}, t) f(y_{i}|t)$. It is easy to show, then, that $\text{E}_{\delta}Y_{t, i} = \pi^{-1}_i \sigma_{Y_{t, i}, \text{E}(\delta_i|Y_{t, i})} + \text{E}Y_{t, i}$. Here, the notation $\sigma_{Y_{t, i}, \text{E}(\delta_i|Y_{t, i})}$ denotes the covariance: $\text{E} \{ Y_{t, i} \text{E}(\delta_i|Y_{t, i}) \} - \text{E}Y_{t, i} \pi_i$. It is also easy to show---insofar as $E(\delta_i| y_i, t) > 0$ for all $y_i$ and $t$---that $f_{\delta}(y|t)$ is supported on the same set as $f(y_i|t)$.

For conciseness, we often denote $Y_{t, i} \in \zeta$ as $Y_{t, \delta_i}$ under the implicit assumption that $\delta_i = 1$. We also use $Y_{T_i} \in \zeta$ or $Y_{T, \delta_i}$ with the understanding that $T$ is fixed to whatever value it takes for unit $i \in I$. We now specify our short list of assumptions more formally, with a re-statement of C1-C3 for clarity.
\begin{enumerate}
\item[C1:] $Y_{t, \delta_i} = Y^t_{t, \delta_i}$ for all $t \in \mathcal{S}_T$ and $\delta_i$ (Consistency)
\item[C2:] Let $\mathbf{L}$ be a vector of random variables. Then $\text{E}Y_t = \text{E}Y^t$ or, conditional on $\mathbf{L}$, $\text{E}Y_{t, \mathbf{L}} = \text{E}Y_{\mathbf{L}}^t$ (Mean exchangeability)
\item[C3:] $0 < \text{Pr}(T_i = 1 | \mathbf{L}) < 1$ for $\forall i \in I$ (Positivity)
\item[C4:] The support of $Y_{T, \delta_i}$ and $Y^{T_i}$ are the same, i.e., for all $Y_{T_i} \in \zeta$, $\mathcal{S}_{Y_{T, \delta_i}} = \mathcal{S}_{Y^{T_i}}$. This can also be stated conditional on $\mathbf{L}$
\item[C5:] Let $Z_{\delta_i}$ be a random variable and say that $\mathcal{L}_n =(I, E_n)$ is an undirected graph with node set $I$ and link set $E_n$ s.t. a link $e_{i, j} \in E_n$ between two nodes $i, j \in I$ is present if and only if $\sigma_{Z_{\delta_i}, Z_{\delta_j}} \neq 0$. Then the mean degree of this graph $n^{-1} \sum_{i=1}^n \sum_{j =1 }^{n-1} 1_{e_{i, j} \in E_n} =\mu_n =o(n)$, where $1_{e_{i, j} \in E_n} = 0$ when $i=j$ by convention and each indicator variable is non-stochastic
\end{enumerate}
The usual provisos that referenced mathematical objects exist is mostly omitted. Reiterating the meaning of C4 is useful as a stepping stone to further consideration. Put succinctly, when C4 holds, it means that the counterfactual distribution and conditional distribution possess the same support, either conditional on some $\mathbf{L} = \mathbf{l}$ or marginally. Rejecting this notion is equivalent to positing that certain values in the support of $Y^T$ ($Y_{\mathbf{l}}^T$) can never materialize with $Y_T$ ($Y_{T, \mathbf{l}}$). This is a strong assertion with non-trivial epistemic consequences, especially in the context of non-informative sampling. If it is believed that C4 cannot be obtained, then those values that exist in the counterfactual support alone have no real world meaning. In this circumstance, we can simply condition on those that can materialize at no empirical loss. 

Notably, sufficient conditions for C4 to hold are the existence of a possibly unmeasured and unknown composite confounder $U$ s.t. $Y^t$ is conditionally independent of $T$ provided $U$, C1 is also true, and the stronger form of sampling positivity holds. These conditions are articulated without further conditioning on $\mathbf{L}$ at no loss of generality. For demonstrative purposes, we prove that these statements imply C4 informally for the case s.t. $U$ is absolutely continuous, also at no loss. To this end, observe the following identity under the first stated premise that $U$ and the referenced densities exist:
$$f(y^{t}) = \int_{\mathcal{S}_U} f(y^{t} | t, u) f(u) du = \int_{\mathcal{S}_U} f(y | t, u) f(u) du$$
Now, note that the following statement is also true:
$$f(y|t) = \int_{\mathcal{S}_U} f(y| t, u) f(u | t)du$$
Since both $f(y^t)$ and $f(y|t)$ are functions of $y$ alone and otherwise share in $f(y| t, u)$ as a basis, it is then implied that $f(y^t)$ and $f(y|t)$ are strictly positive on the same set of values. Furthermore, since $E(\delta | y, t) > 0$ implies that $f_{\delta}(y|t) > 0$ on the same set of $y$ values s.t. $f(y|t) > 0$, by transitivity, $f_{\delta}(y|t)$ also shares the same support with $f(y^t)$. This supplies C4. The reason we chose to treat C4 as an assumption, however, is because doing so does not require the existence of a random variable $U$ with the stated properties. It is therefore feasible to achieve C4 in more general circumstances. Nevertheless, the conditions that imply C4 are also relatively mild and employed on a regular basis. For instance, every condition except one---the stronger form of sampling positivity---is required by methodologies that estimate expected causal effects via C1-C3. The stronger form of sampling positivity is often assumed implicitly. As a consequence, if one believes that it is even possible to try and identify the expected causal effect, then the functional average treatment effect is already identified in many circumstances.

The last assumption, C5, is required for establishing statistical consistency. Results are often proven under the premises of mutual independence and non-informative sampling. This restricts their utility, especially since many modern research settings depend upon non-probability samples of outcome variables that partake in complicated and unknown systems of possibly 'long-range' probabilistic dependence. Furthermore, this is also restricting since informative sampling can induce statistical dependencies. By proving our results under more general conditions, we expand their reliability into these contexts. C5 essentially asserts that the mean number of outcome variables in a sample that a typical one is correlated with is sub-linear in $n$, i.e., that $n^{-1} \mu_n \to 0$ as $n \to \infty$. Note that this is a very mild assumption since it still allows for the mean number of statistical dependencies present in the sample to diverge with sample size. It places no additional constraint on the exact form of the probabilistic dependencies. We also reference an alternative: C5'. This assumption is exactly the same, except it makes use of a dependence graph s.t. a link exists between two nodes if and only if their corresponding outcome variables are statistically dependent.

Our first proposition establishes functional average exchangeability. Although it is trivial mathematically, it provides a useful foundation. 
\begin{proposition}{(\textbf{Functional Average Exchangeability})}
Suppose C1 and C4. Then $\text{Av}(Y_{T, \delta}) = \text{Av}(Y^{T})$.
\end{proposition}
\begin{proof}
The result follows directly from the premises. Since $\mathcal{S}_{Y_{T, \delta}} = \mathcal{S}_{Y^T}$, $\int_{\mathcal{S}_{Y_{T,\delta}}} 1 \cdot dy = \int_{\mathcal{S}_{Y^T}} 1 \cdot dy = R$ WLOG for the continuous case and $\int_{\mathcal{S}_{Y_{T,\delta}}} y dy = \int_{\mathcal{S}_{Y^T}} y dy$.
\end{proof}
Proposition 1 is extendable to $\text{Av}_{\mathbf{x}}(Y^T)$ or to the conditional case s.t. there exists some vector $\mathbf{L}$ where $\text{Av}(Y_{T, \mathbf{L}, \delta}) = \text{Av}(Y_{\mathbf{L}}^{T})$. However, again, this is omitted.

Before continuing, a common caveat is due. Identifying a counterfactual parameter statistically is not equivalent to the identification of a causal one. Causal relationships cannot be inferred from statistical relationships alone \cite{pearl2003statistics}. A theory of causation---perhaps represented by a structural causal model---is therefore still required if causal meanings are to be supplied to the functional average \cite{pearl2010causal, hernan2010causal}. These results only simplify this process for a small set of related parameters by removing the strict need to measure a proper set of adjusting variables in some circumstances. In other words, if C1 and C4 hold, then, provided a structural causal model, the functional average effect can be identified even without accounting for unmeasured confounders.  

\subsection{The Problem of Estimation}
The simplicity of Proposition 1 and the relative mildness of C4 unfortunately coexist with the difficulty of estimating $\text{Av}(Y_{T, \delta})$. Here, the `no is no free lunch' addage returns. A theoretical estimator can be constructed, nevertheless, using the following two identities. We tacitly condition on $1_{\mathcal{S}_Y}$ for ease of reading: $\text{E}\{ f^{-1}(Y)Y \} = \int_{\mathcal{S}} y dy$ and $\text{E}\{ f^{-1}(Y) \} = R$. This naturally suggests an estimator of the following form:
\begin{equation}
\tilde{Av}(Y_{\delta}) = \{ \sum^n_{i=1} f_{\delta_i}^{-1}(Y_{\delta_i})\}^{-1} \sum^n_{i=1} f_{\delta_i}^{-1}(Y_{\delta_i}) Y_{\delta_i}
\end{equation}
In this section, we investigate some of the features of plug-in estimators for eq. (2.1). After exploring the discrete case, we offer brief commentary on the difficulties of the continuous one. Then we re-visit the sample mid-range estimator. After completing these explorations, we introduce a bootstrapping strategy for conducting inference.

\subsubsection{Discrete Estimators of $\tilde{Av}(Y_{\delta})$}
When $Y_{\delta_i}$ is discrete and $\hat{f}_{n} (y) = n^{-1} \sum_i^n 1_{Y_{\delta_i} = y} = n^{-1} M_y$, the empirical plug-in for eq. (2.1) reduces to an intuitive estimator. Say $1_{y \in \mathcal{S}_{\zeta}}$ is an indicator that a value $y \in \mathcal{S}_Y$ is observed and therefore in the support of the empirical distribution: $\mathcal{S}_{\zeta}$. Then the empirical plug-in for eq. (2.1) reduces to $\hat{Av}(Y_{\delta}) = \{ \sum_{y \in \mathcal{S}_Y} 1_{y \in \mathcal{S}_{\zeta}} \}^{-1} \sum_{y \in \mathcal{S}_Y} 1_{y \in \mathcal{S}_{\zeta}} y$ under the convention that $\hat{f}^{-1}_n (y) =0$ when $y \notin \mathcal{S}_{\zeta}$. To see this, observe that for an arbitrary set of materialized sample values in $\zeta$, where $\zeta$ is temporarily treated as a set of constants, $\sum_{y \in \zeta} \hat{f}_n^{-1} (y) = n\cdot |\mathcal{S}_{\zeta}|$ and $\sum_{y \in \zeta} \hat{f}_n^{-1} (y) y = n \cdot \sum_{y \in \mathcal{S}_{\zeta}} y$. Therefore, the discrete plug-in for eq. (2.1) is simply the arithmetic average of the unique values observed in the sample.

We now establish the statistical consistency of this plug-in under general dependency conditions. For the rest of this section, we omit notation for $\delta$ with the understanding that it is implicit whenever we are dealing with sampled outcomes. To this end, note that $\text{Pr}(Y_1 \neq y, Y_2 \neq y, \ldots, Y_n \neq y) = \text{Pr}(Y_n \neq y | Y_{n-1} \neq y, \ldots, Y_1 \neq y) \cdot \text{Pr}(Y_{n-1} \neq y | Y_{n-2} \neq y, \ldots, Y_1 \neq y) \cdots \text{Pr}(Y_1 \neq y)$ and define a corresponding sequence $\mathcal{F} =(\text{Pr}(Y_i \neq y| \mathcal{A}_i))_{i \in I}$ under the convention that $\text{Pr}(Y_1 \neq y| \mathcal{A}_1) = \text{Pr}(Y_1 \neq y)$.
\begin{proposition}
Suppose a sample $\zeta = \{ Y_i \}_{i \in I}$. Observe $\mathcal{F} =(\text{Pr}(Y_i \neq y| \mathcal{A}_i))_{i \in I}$ as previously defined and say $k(n) =|\{ s \in \mathcal{F} | s < 1 \}|$, where $s \in \mathcal{F}$ indicates here that $s$ is present in the sequence. If $k(n) \to \infty$ as $n \to \infty$, then $\hat{\text{Av}}(Y) \overset{a.s.}{\to} \text{Av}(Y)$ as $n \to \infty$, where $\overset{a.s.}{\to}$ denotes almost sure convergence.
\end{proposition}
\begin{proof}
Let $y \in \mathcal{S}_Y$ be arbitrary and denote $\mathcal{S}_{\zeta} \subseteq \mathcal{S}_Y$ as the set of observed values. Then $\text{Pr}(y \in \mathcal{S}_{\zeta}) = 1- \text{Pr}(y \notin \mathcal{S}_{\zeta}) = 1- \text{Pr}(Y_1 \neq y, Y_2 \neq y, \ldots, Y_n \neq y) = 1-\text{Pr}(Y_n \neq y | Y_{n-1} \neq y, \ldots, Y_1 \neq y) \cdot \text{Pr}(Y_{n-1} \neq y | Y_{n-2} \neq y, \ldots, Y_1 \neq y) \cdots \text{Pr}(Y_1 \neq y)$. Now, suppose $k(n)$ probabilities in the sequence $\mathcal{F}$ are strictly less than one. Denote the maximum of these probabilities as $\text{Pr}(Y_* \neq y)$ and note that since $\text{Pr}(Y_* \neq y) < 1$, there exists some $\epsilon > 0$ s.t. $\text{Pr}(Y_* \neq y) = 1-\epsilon$. Then:
\begin{align*}
\text{Pr}(y \notin \mathcal{S}_{\zeta}) & = \prod_{i=1}^n \text{Pr}(Y_i \neq y | \mathcal{A}_i) \\
& \leq 1^{n-k(n)} \cdot \{ \text{Pr}(Y_* \neq y) \}^{k(n)} \\
& =  \{ 1-\epsilon \}^{k(n)}
\end{align*}
Hence:
$$0 \leq \lim_{n \to \infty} \text{Pr}(y \notin \mathcal{S}_{\zeta}) \leq \lim_{n \to \infty} \{ 1-\epsilon \}^{k(n)} =0 $$
This of course implies that $\text{Pr}(y \in \mathcal{S}_{\zeta}) \to 1$ as $n \to \infty$. Next, define an indicator variable $1_{y \in \mathcal{S}_\zeta}$ and also $Z_{k_*} = \underset{k > n}{\text{sup}}|1_{y \in \mathcal{S}_{\zeta_k}} - \text{Pr}(y \in \mathcal{S}_{\zeta_k})| = |1_{y \in \mathcal{S}_{\zeta_{k_*}}} - \text{Pr}(y \in \mathcal{S}_{\zeta_{k_*}})|$. Letting $\epsilon > 0$ be arbitrary again:
\begin{align*}
\text{Pr}(Z_{k_*} > \epsilon) \leq \epsilon^{-2} \{ \text{Pr}(y \in \mathcal{S}_{\zeta_{k_*}}) \cdot (1-\text{Pr}\{ y \in \mathcal{S}_{\zeta_{k_*}} \})  \}
\end{align*}
This then implies that:
$$\lim_{n \to \infty} \text{Pr}(Z_{k_*} > \epsilon) \leq \lim_{n \to \infty} \epsilon^{-2} \{ \text{Pr}(y \in \mathcal{S}_{\zeta_{k_*}}) \cdot (1-\text{Pr}\{ y \in \mathcal{S}_{\zeta_{k_*}} \})  \} =0  $$
Hence, $1_{y \in \mathcal{S}_{\zeta}} \overset{a.s.}{\to} 1$. Thereby, since $y$ was arbitrary, it is then implied that $\hat{\text{Av}}(Y) = \{ \sum_{y \in \mathcal{S}_Y} 1_{y \in \mathcal{S}_{\zeta}} \}^{-1} \sum_{y \in \mathcal{S}_Y}  1_{y \in \mathcal{S}_{\zeta}} y \overset{a.s.}{\to} R^{-1}  \sum_{y \in \mathcal{S}_Y} y = \text{Av}(Y)$ since $R$ is finite.
\end{proof}
The elementary nature of $\hat{\text{Av}}(Y)$ makes it a reliable estimator for relatively simple outcome variables. $\hat{\text{Av}}(Y)$ will converge almost surely at an unknown, but very fast rate in all likelihood, and even in the presence of stark probabilistic dependencies, when the scale of the outcome variable possesses a small number of unique values. This statement obviously applies to sample extremes in addition.

Unfortunately, however, quantifying the rate of convergence---or the uncertainty associated with finite sample estimates---is difficult. This is true even under mutual independence. To appreciate this, it is sufficient to observe $\sum_{y \in \mathcal{S}_Y} 1_{y \in \mathcal{S}_{\zeta}} y$. Since $\text{E} 1_{y \in \mathcal{S}_{\zeta}} = p_{y, n}$ is unknown, so is $\text{Var} \{ \sum_{y \in \mathcal{S}_Y} 1_{y \in \mathcal{S}_{\zeta}} y \} = \sum_{y \in \mathcal{S}_Y} p_{y, n} \cdot (1-p_{y, n}) \cdot y^2$. If $\zeta$ is a sample of identically distributed and mutually independent outcome variables, then we can attempt to estimate $p_{y, n}$ with $\hat{p}_{y, n} = 1-\{ 1-\hat{f}_n (y) \}^n$. However, $\hat{p}_{y, n} \equiv 0$ when $y \notin \mathcal{S}_{\zeta}$. Furthermore, when $y \notin \mathcal{S}_{\zeta}$, it is also unknown by definition. Hence, reasonably estimating  $\text{Var} \{ \sum_{y \in \mathcal{S}_Y} 1_{y \in \mathcal{S}_{\zeta}} y \}$ requires knowledge that makes estimating $\text{Av}(Y)$ arguably redundant.

A recourse to the central limit theorem is also unavailable. This is because $|\mathcal{S}_Y|$ is finite by construction. Therefore, $\hat{Av}(Y)$ will always be a finite sum of random variables. In some circumstances---such as when $|\mathcal{S}_Y|$ is reasonably large---a normal approximation might still function with an acceptable degree of accuracy. However, for reasons already explored, this strategy will still require a strong set of assumptions about sampling probabilities and potential values. 

\subsubsection{Continuous Outcomes}

Kernel density estimation is an intuitive choice to estimate eq. (2.1) for the continuous case. However, the properties of this plug-in are also largely intractable and unknown. For example, although the properties of a kernel density estimator $\hat{f}_n (y)$ are well-researched for a constant $y \in \mathcal{S}_Y$ \cite{hansen2008uniform, chen2017tutorial, zambom2013review}, the behavior of  $\hat{f}_n (Y)$, i.e., the random variable defined and evaluated on the same random outcome that was utilized to construct it, is not as well-studied. This is because kernel density estimation is often evaluated on a grid of deterministic points. Establishing the asymptotic properties of a statistic of the form $\{ \sum_{i=1}^n \hat{f}_n^{-1} (Y_i) \}^{-1} \sum_{i=1}^n \hat{f}_n^{-1} (Y_i) \}^{-1} Y_i $, where $\hat{f}_n^{-1} (Y_i) = nh \cdot  \{ \sum_{j=1}^n K\{h^{-1} (Y_i - Y_j) \} \}^{-1}$ for some $h > 0$ and kernel function $K(\cdot)$, although promising, is therefore also non-trivial. Such considerations also require a detailed consideration of possible kernel functions. Since---in general---we are interested in establishing statistical consistency under very general dependency conditions, we avoid this enterprise in this manuscript.

We also avoid other options for density estimation since they arrive with similar challenges, some as of yet undisclosed. For instance, estimators that use reciprocal estimated densities can possess unstable variances when the underlying distribution possesses a density that decays smoothly toward zero. Moreover, since each $\hat{f}_n (y)$ is typically a function of the entire sample, plug-ins for eq. (2.1) will necessarily possess a myriad of complex dependencies. This will prevent any elementary citation of a central limit theorem. Just as importantly, it will also limit the applicability of concentration inequalities for finite sample inference. Hence, although this is a promising area of research that demands attention, no further consideration is offered here.
\newline
\paragraph{\textit{Mid-range estimation}} For a large special class of bounded random variables, the mid-range is a simple alternative for estimating functional averages, including those from continuous distributions. Let $Y_{(i)}$ for $i \in I_n = \{1, \ldots, n \}$ denote the $i$th order statistic of a sample s.t. $Y_{(1)} \leq Y_{(2)} \leq \cdots \leq Y_{(n)}$. The mid-range, $\hat{MR}\{ Y \}$, or simply $\hat{MR}$ when convenient, is defined as follows: $\hat{MR} \{ Y \} = 2^{-1} \{ Y_{(1)} + Y_{(n)} \}$.

Naturally, the sample mid-range estimates the population mid-range $MR=2^{-1}(m + M)$. Linear combinations of order statistics are also well-studied \cite{chernoff1967asymptotic, hosking1990moments, bickel1973some, david2004order}. However, the sample mid-range is often ignored, and especially in applied settings, because of its possible inefficiency and since its distribution also admits no closed-form expression in a majority of settings. Before offering an exposition on some of its properties, we offer a useful definition, which highlights our interest in it. We say a random variable is \textit{regular} when it is supported on a single interval of real numbers or a complete subset of integers. This definition is helpful because $\text{Av}(Y) = MR \{Y \}$ when $Y$ is a regular random variable. 
\begin{definition}
A random variable will be said to be regular if and only if its support $\mathcal{S}$ is a single interval of real numbers or a complete subset of integers starting at some $m \in \mathbb{N}$ and ending with a maximum integer $M$ s.t. if integer $c \in \mathcal{S} \setminus{M}$, then $c+1 \in \mathcal{S}$.
\end{definition}
$\hat{MR}$ is a statistically consistent estimator of $MR$ for outcome variables with finite support under the assumption of mutual independence. Barndorff-Nielsen \cite{barndorff1963limit} established sufficient and necessary conditions for the statistical consistency of extreme order statistics. Almost sure convergence, and therefore also convergence in probability ($\overset{p}{\to}$), of an extreme order statistic to its asymptotic target is trivially fulfilled when there exists a $y \in \mathcal{S}_Y$ s.t. $F(y)=1$ and $F(y - \epsilon) < 1$ for all $\epsilon > 0$. Hence, $\hat{MR}$ also converges almost surely to its population value for all bounded distributions under this setup. Sparkes and Zhang \cite{u2023properties} extended this result to much more general scenarios of statistical dependence. If we define a sequence of conditional cumulative distribution functions (CDFs) $\mathcal{F}$ in the same spirit as Proposition 2, it can be demonstrated that extreme order statistics converge in probability to their target values for bounded random variables insofar as the number of conditional CDFs in $\mathcal{F}$ that is strictly less than unity diverges as $n$ becomes arbitrarily large. This is once again a very mild assumption since the dependencies involved can induce arbitrary changes in the behaviors of the distribution functions otherwise. Since the random variables considered here are bounded, convergence in probability of the sample extremes also implies their almost sure convergence.

Nevertheless, as previously mentioned, when the distribution of $Y$ is unknown, no reliable expression for the distribution of $\hat{MR}$ is accessible to use for inference: a situation that is analogous to the discrete plug-in estimator for eq. (2.1). Although extreme value theory helps to address this issue under the assumption of an independent and identically distributed sample or a stationary sequence of outcome variables \cite{leadbetter1988extremal}, it is insufficient without additional parametric constraints. For instance, classical results establish, for a suitable sequence of constants $a_n$ and $b_n$, that $a^{-1}_n \{ Y_{(n)} - b_n \}$ converges weakly to one of three distributions under certain regularity conditions \cite{leadbetter1988extremal, smith1990extreme, haan2006extreme, kotz2000extreme}. These are the Gumbel (type I), Fréchet (type II), and reverse Weibull (type III) distributions. These results are also sufficient for reasoning about the sample minimum since $Y_{(1)}=-\underset{i\in I_n}{\text{max}}(-Y_i)$.  Bingham \cite{bingham1995sample, bingham1996sample} uses the convergence of  $a^{-1}_n \{ Y_{(n)} - b_n \}$ to a type I or III extreme value distribution and the asymptotic independence between $Y_{(1)}$ and $Y_{(n)}$ to derive limiting distributions of $\hat{MR}$ when its underlying distribution function is also symmetric. However, the expressions derived for these asymptotic distributions ultimately depend upon unknown normalizing constants that are specific to the marginal distribution function of the sampled outcomes. Broffitt \cite{broffitt1974example} and Arce and Fontana \cite{arce1988midrange} provide similar explorations for $\hat{MR}$ under the auspices that $\zeta$ is an identical and independent sample from a symmetric power law distribution. Under this constraint, the limiting variance of $\hat{MR}$ is derived as $\{12 \cdot a^2 \cdot  \{ \log{(n)} \}^{2(1-1/a)} \}^{-1} \cdot \pi^2$ for some distribution-specific $a > 0$, for instance. Unlike the sample mean under these same conditions, though, an expression for a governing probability law that does not depend on the marginal distribution functions in question is again unavailable, even asymptotically. This situation extends to the asymmetric and non-independent cases, which are even more poorly studied. 

Bootstrapping is a feasible option for inference, provided these challenges. However, it is also not without problems. Traditional bootstraps condition on the observed values of $\zeta$ and use the strong consistency of the empirical CDF $\hat{F}_n(x)$ to emulate the sampling distribution of a statistic of interest via a re-sampling procedure \cite{efron1994introduction}. They require that the statistic of interest, say $T_0(Y_1, Y_2, \ldots, Y_n)$, is a well-behaved functional of the marginal CDF and that the targeted parameter is not a boundary value of the support \cite{bickel1981some}. Overall, bootstrapping processes usually behave as intended under the same set of conditions that supply a central limit theorem. For these reasons, traditional bootstraps are problematic for functions of extreme order statistics. The $m$-out-of-$n$ bootstrap, however, has proven to be an effective procedure in this domain \cite{bickel2008choice}.  Essentially, a basic $m$-out-of-$n$ bootstrapping process re-samples $m$ observations from $\zeta$ with or without replacement s.t. $n^{-1}m\to 0$ as $m \to \infty$. It can provide approximately valid inference when traditional methods fail. See Bickel and Ren \cite{bickel2001bootstrap}, Swanepoel \cite{swanepoel1986note}, Beran and Ducharme \cite{beran1991asympotic}, or Politis, Romano, and Wolf \cite{politis2001asymptotic} for additional background and resources on the topic. Pertinently, the $m$-out-of-$n$ bootstrap is also capable of handling situations with dependent observations insofar as an appropriate sub-sampling strategy is used.

Nevertheless, the $m$-out-of-$n$ bootstrap (and other forms of bootstraps for dependent observations) is still insufficient for the context and conditions of this paper. Three reasons substantiate this claim. Firstly, we require a version of the bootstrap that is capable of reliably capturing $\theta$ under fairly general but unknowable dependency conditions. This rules out approaches such as the $m$-out-of-$n$ bootstrap, or bootstrapping processes such as the block bootstrap, which require a re-sampling theory that corresponds adequately to the unknown dependency structure, and which typically exclude the existence of long-range dependencies \cite{ shao2010dependent, hall1995blocking, kreiss2011bootstrap, lahiri2003resampling}. 

Secondly, we are interested in reasoning about $\theta = \text{Av}(Y)$ and not $E\{ \hat{MR}\}$ for a particular $n$. In most circumstances where $\hat{MR}$ will be used, i.e., those circumstances s.t. the marginal distributions are not symmetric, it will be a biased estimator \cite{arce1988midrange}. It is likely that $\theta$ rests on the boundary of the support of $\hat{MR}$ in these circumstances. Convergence to $\theta$ might be slow and characterized by an unknown rate in addition \cite{broffitt1974example}. Consequently, any inferential procedure for $\hat{MR}$ must be flexible enough to provide cogent statements about $\theta$---and not simply about $\text{E} \{\hat{MR} \}$ at a particular value of $n$---and even when sample sizes are modest. This necessitates conservative approaches for inference that allow for $\theta$ to sit outside of the empirical distribution of the bootstrapped statistics. It also therefore rules out popular bootstrapping methodologies, which construct confidence sets that are subsets of this observed range. Lastly, we wish to use a bootstrapping strategy that does not rely on the assumption that $T_0$ is a smooth functional of $F(x)$.

Further work to produce more efficient closed-form approximations is of course a preferable route. Since the mid-range is more efficient than the sample mean when non-negligible probability rests in the extremes of the support \cite{rider1957midrange}, it can provide a more efficient estimator of expected causal effects in many circumstances: a fact that is often neglected. Overall, however, due to the complicated probabilistic character of order statistics, this is a onerous road that possesses no immediate destination, especially when outcome variables are dependent and their joint distribution is unknown. This ultimately necessitates a different type of bootstrapping strategy.
\subsubsection{The Hoeffding Bootstrap}
With these prior facts in mind, we offer two limited, but related solutions, although only the first is discussed in this section. In summary, we assert that the bootstrap can be re-purposed to construct conservative confidence sets under fairly general conditions of statistical dependence and under milder regularity conditions. Notably, this re-purposed bootstrap, which we call the Hoeffding bootstrap, can be applied to all functional average estimators previously explored.

We now provide a synopsis of the first approach. Further details and the proof are provided in the supplementary materials. Essentially, we show that (1) if a statistician does not condition on the observed values of $\zeta$ and treats each re-sampled $Y_i$ as random, (2) if the maximum order statistic of the bootstrapped statistics is a discrete or an absolutely continuous random variable (at least asymptotically), and (3) if the outcome variables being re-sampled are not monotonic transformations of one another, say, or they do not partake in other forms of truly extreme statistical dependence, then the estimator for the range of the bootstrapped statistics is a statistically consistent estimator of a value that is greater than or equal to the range of $T_0(Y_1, Y_2, \ldots, Y_n)$ as $n$ and the number of bootstraps become arbitrarily large. Call this estimator $\hat{M} - \hat{m}$ in relation to $M$ and $m$, which now designate the maximum and minimum of the support of the bootstrap distribution. Insofar as $T_0(Y_1, Y_2, \ldots, Y_n)$ has finite support, the estimator of the bootstrap range can then be used in conjunction with Hoeffding's inequality to produce large-sample confidence intervals for $\theta$ with at least $1-\alpha$ coverage of the form $T_0(Y_1, Y_2, \ldots, Y_n) \pm \{ \hat{M} - \hat{m}\} \sqrt{2^{-1} \text{log}(2/\alpha)}$. The performance of confidence sets constructed with this strategy are evaluated in Section 4 and also in the supplementary materials. For clarity, we provide a schematic of the process below:
\begin{enumerate}
\item[I.] Acquire a sample of random variables $\zeta = \{ Y_i \}_{i \in I_n}$ and compute a statistic $T_0(Y_1, \ldots, Y_n)$
\item[II.] Draw $m \leq n$ random variables from $\zeta$ with or without replacement via a simple random sample or a theoretically guided process that attempts to reproduce a dependency structure. Compute the new statistic $T_1$ from these variables
\item[III.] Repeat I. and II. $K(n) -1$ times, where $K(n)=K$ is reasonably large, and construct $\{ T_k \}_{k \in \mathcal{K}}$ for $\mathcal{K} = \{0, 1, 2, \ldots, K \}$
\item[IV.] Set $\hat{M} - \hat{m} = \underset{k \in \mathcal{K}}{\text{max}}(T_k) - \underset{k \in \mathcal{K}}{\text{min}}(T_k)$
\item[V.] Construct an estimate of an at least $1-\alpha$ confidence set with $T_0 \pm \{ \hat{M} - \hat{m} \} \sqrt{2^{-1} \text{log}(2/\alpha)}$
\end{enumerate}
Requiring the maximum order statistic of the bootstrap sample to possess a density when it is continuous is non-trivial and might limit the applicability of the approach. However, the stipulation can be feasibly checked by observing the histogram of the bootstrap distribution. If it possesses a smooth shape without too many jagged breaks in its continuity across the x-axis, this is at least a good sign. Nevertheless, this limitation is addressed in Section 3. Although the provisional solution offered there does not strictly require $\mathcal{U}$ concepts, which we also introduce in Section 3, they provide clarity on the topic. Essentially, we show that it is still probably safe to use a slightly more conservative form of the same confidence set when a density does not exist, or even when the sample bootstrap range is not a statistically consistent estimator of $M-m$.

Moreover, even if the conditions that validate this approach are not met, it is apropos to state that the Hoeffding bootstrap will always perform better than strategies that use bootstrapped $t$-statistics or bootstrapped normal approximations. This fact essentially flows from Popoviciu's inequality, which states that $\text{Var}(T_0) \leq 4^{-1}(M_0 - m_0)^2$ when $T_0$ is bounded. This inequality also applies to empirical distributions. For instance, observe when $\alpha = .05$. Then $1.96 \cdot S_{T_k} \leq (\hat{M} - \hat{m}) \leq \hat{M} - \hat{m} \cdot \sqrt{2^{-1}\text{log}(2/\alpha)} \approx 1.35 \cdot (\hat{M} - \hat{m})$, where $S_k$ is the sample standard deviation of the bootstrap distribution.

This last fact partially motivates the use of Hoeffding's inequality. Note that if $m \leq T_0 \leq M$, it is already implied that $\text{E}T_0 \in [m, M]$ and hence that, for some sufficiently large $n$ and $K$, $\text{E}T_0 \in [\hat{m}, \hat{M}]$ with probability that is approximately one. However, since we actually want to capture $\theta$ and, in practice, we often use only moderately sized $K$ with moderate $n$, Hoeffding's inequality supplies an intuitive and well-established interval that already arrives with a penalty that is adjusted by $\alpha$. Insofar as $T_0 \overset{a.s.}{\to} \theta$ as $n \to \infty$, using Hoeffding's inequality as a penalty also asymptotically guarantees at least $1-\alpha$ coverage for $\theta$, as previously mentioned. A method for constructing confidence intervals such as this, although conservative, avoids dependency modeling and applies to a much larger class of statistics.

\section{$\mathcal{U}$ Random Variables and Counterfactual Linear Regression}
In this section, we discuss a class of random variables---the $\mathcal{U}$ class---that can help us avoid the difficulties encountered in Section 2. Recall: although we established that functional average causal effects can be identified and statistically consistently estimated under very mild assumptions---and even without adjusting for confounders---establishing efficient methods of statistical inference for these estimators is challenging. Ultimately, this is because the plug-in estimators defined and the sample mid-range possess largely intractable properties, and even asymptotically, in the absence of additional constraints that are in all likelihood inappropriate for applied settings. These difficulties are removed when working with $\mathcal{U}$ random variables as outcomes since they ultimately allow for the functional average to be estimated by standard additive statistics with well-known properties. We also show that $\mathcal{U}$ random variables are important because they can imbue basic linear regressions and analyses of variance with counterfactual---and thus possibly causal---interpretations under conditions traditionally assumed for estimating associations. On a similar note, we also prove that properties of $\mathcal{U}$ variables can be used to establish sufficient conditions for mean exchangeability and that they can be used to defend an extension of the Hoeffding bootstrap. First, however, a definition of a $\mathcal{U}$ random variable is helpful. We assume that all integrals and mathematical objects exist when referenced, as per usual.
\begin{definition}
Let $g$ be a measurable function. A random variable $g(Y)$ will be said to be in the class of $\mathcal{U}$ random variables if and only if $\text{E} \{ g(Y) \} = \text{Av}\{ g(Y) \}$. Similarly, the same will be said w.r.t. $\mathbf{X} \in \mathbb{R}^k$ for $Y=g(X_1, \ldots, X_k)$ if and only if $\text{E}Y = \text{Av}_{\mathbf{x}}(Y)$.
\end{definition}
Definition 2 stipulates that a random variable is $\mathcal{U}$ class w.r.t. some space when its expected value is equal to its average functional value in that space. Stated in a probabilistic fashion, a variable $Y$ is $\mathcal{U}$ class if one can take its expectation w.r.t. a uniform measure without changing its value. This type of random variable is ubiquitous in practice. A host of their properties has been investigated elsewhere \cite{u2023properties}. To familiarize the reader, we provide a list of important ones in \textbf{Table 1}. Essentially, $\mathcal{U}$ variables are closely related in concept to sum-symmetry of the CDF and structured but uncorrelated deviations from uniformity. All bounded and symmetric random variables are in the $\mathcal{U}$ class, for instance, although symmetry is not a necessary condition. Continuous and regular random variables with densities that are proportional to their standard deviation behave more and more like $\mathcal{U}$ random variables as $n \to \infty$ if their variance tends to zero. In an abuse of notation, we will say $Y \in \mathcal{U}$ if $Y$ is in this class of random variable.
\begin{table}[H]
\centering
\caption{Basic Properties of $\mathcal{U}$ Variables}
\begin{tabularx}{\textwidth}{p{5cm}|c| p{4cm}}
\hline
Property & Variable Type & Conditions and Definitions \\
\hline
$\sigma_{Y, f^{-1}(Y)} = 0$ & A & $Y \sim f(y)$ \\
$Y \in \mathcal{U} \implies cY \in \mathcal{U}$ and $Y + c \in \mathcal{U}$ & A & $c \in \mathbb{R}$ \\
$\text{E}\{f(Y)\} = R^{-1}$ implies $Y \in \mathcal{U}$ & A & --- \\
$Y \in \mathcal{U}$ is equivalent to $F(Y) \in \mathcal{U}$ & R, C & $F(y)$ is CDF of $Y$ \\
$\int_{m}^M F(y) dy = \int_m^M S(y) dy$ & R, C, D & $S(y) =  1-F(y)$, $\mathcal{S}_Y = [m, M]$ \\
$U = Y + \epsilon$ s.t. $\text{E}(\epsilon | Y) = 0$ & R, C & $U \sim Unif(m, M)$, $f(y) \leq R^{-1}$ in left tail and unimodal \\
$\text{Pr}(|S_n - \text{E}S_n| > \epsilon) \leq 2 \text{exp} \{ - \{ \sum_{i=1}^n R_i^2 \}^{-1} 6 \epsilon^2  \}$ & R, C & $S_n = \sum_{i=1}^n Y_i$, $Y_i \in \mathcal{U}$, $C^*$ \\
$\sum_{i =1}^M F(i) = \text{E}Y$ & R, D & $\mathcal{S}_Y = \{1, 2, \ldots, M \}$ \\
$\sum_{i =1}^M F(i) - \sum_{i =1}^M S(i) =1$ & R, D &  $\mathcal{S}_Y = \{1, 2, \ldots, M \}$ \\
\hline
\end{tabularx}
\caption*{\small A = All, R = Regular, C = Continuous, D = Discrete, $C^* = \text{max}\{ \text{E} (\text{exp} \{s S_n \} ), \text{E} (\text{exp} \{-s S_n \}) \} \leq \text{Av}_{\mathbf{y}} ( \text{exp} \{s S_n \} ), s >0$}
\end{table}
Out of these properties, we draw special attention to the concentration inequality: $\text{Pr}(|S_n - \text{E}S_n| > \epsilon) \leq 2 \text{exp} \{ - \{ \sum_{i=1}^n R_i^2 \}^{-1} 6 \epsilon^2  \}$. The condition detailed in the table footnote is very mild and does not require independence. In fact, it can be true even when every single outcome variable in a sample is statistically dependent, insofar as the average correlation between those variables is mild, or $\mu_n$ is bounded if this is not the case. Discussion on this assumption is also available elsewhere \cite{u2023properties}.  Put succinctly, a researcher can expect it to be fulfilled if each $Y_i$ is symmetric---or at least relatively symmetric---and the joint distribution of the sample is biased away from $n$-tuples in the joint support that inflate $S_n$. We note that this is useful since these conditions often apply to the error distributions of statistics of interest, including those of linear regressions. Moreover, if $S_n$ converges in probability to $\text{E}S_n$ as $n \to \infty$, this is also supportive of the notion that the condition is fulfilled for sufficient sample sizes. A proof that $S_n$ converges almost surely to $\text{E}S_n$ under our conditions is provided in the supplementary material.

Next, we introduce two simple propositions with direct practical or theoretical interest for causal inference. Proposition 3 follows directly from our main conditions and solves the problem of estimating $\text{Av}(Y_{\delta, t})$ for variables in the $\mathcal{U}$ class since it allows for the replacement of the estimators of the previous section with the sample mean. Proposition 4 establishes an interesting sufficient condition for mean exchangeability. We only prove these statements for functional averages w.r.t. the range of $Y$. This is for conciseness. Note that all results are easily extended to the excluded case.

Finally, observe that, although the notation was omitted, the following results also apply when the random variables are conditioned upon another vector of random variables $\mathbf{L}$, perhaps to facilitate the fulfillment of C4 or $\mathcal{U}$ status. In this case, we would also assume C3, although this will also be left unmentioned.
\begin{proposition}
Suppose C1 and C4. If $Y_{\delta, t} \in \mathcal{U}$, then $\text{E}Y_{\delta, t} = \text{Av}(Y^t)$.
\end{proposition}
\begin{proof}
The proof is again one line under the premises: $\text{E}Y_{\delta, t} = \text{Av}(Y_{\delta, t}) = \text{Av}(Y^t)$.
\end{proof}
Again, the properties of plug-in estimators for eq. (2.1) are not easy to discern in general and citations of the central limit theorem are also questionable or implausible. However, the properties of $\bar{Y}_{\delta, t}$ are exceptionally well-known. This largely solves the problem insofar as the sampling process secures a sample of $\mathcal{U}$ variables. Again, since C4 only requires the preservation of the support, this allows for an arbitrary distortion of the population distribution otherwise: a fact that is liberating w.r.t. study design and execution.

Note also that Proposition 3 is not as trivial as it seems. It is well known that the sample mean and mid-range estimate the same parameter when the underlying distribution is symmetric. However, it is false that all $\mathcal{U}$ random variables are symmetric. Hence, the $\mathcal{U}$ concept expands the universe where the sample mean can replace the mid-range. The next proposition establishes a new route to justifying the validity of C2, as previously mentioned.
\begin{proposition}
Suppose $Y^t$ and $Y_{\delta, t}$ are both $\mathcal{U}$ random variables under C1 and C4. Then $\text{E}Y^t = \text{E}Y_{\delta, t}$.
\end{proposition}
\begin{proof}
By our premises, the following string of equalities applies: $\text{E}Y^t = \text{Av}(Y^t) = \text{Av}(Y_{\delta, t}) = \text{E}Y_{\delta, t}$.
\end{proof}
Great care and energy of argument are often expended to establish that $\text{E}Y^{t} = \text{E}Y_{t}$. Proposition 4 offers a new manner of doing so insofar as it is believed that the experimental distribution is sum-symmetric. Conditional or unconditional on some vector $\mathbf{L=l}$, insofar as the researcher is willing to posit that the experimental distribution is in the $\mathcal{U}$ class, all that is actually required is a sufficiently executed sampling process that preserves the support and induces \textit{any} form of $\mathcal{U}$ status. Then it is implied that (conditional) mean exchangeability is achieved. Once more, since it seems plausible that $Y_{t}$ or $Y_{t, \mathbf{l}}$ can be mapped into a great number of $\mathcal{U}$ distributions on the same support via different sampling designs or conditioning, this result is potentially very useful.

For instance, if it is believed that the counterfactual distribution is symmetric, then a non-rejection of a statistical hypothesis of symmetry in the observed distribution can be supporting evidence that C2 is fulfilled. More generally, if the distance between the mid-range and the sample mean is small---and here one must be diligent in deciding what precisely defines the quality of this distance---this can also be construed as evidence. A researcher can also observe the behavior of the empirical CDF for visual confirmation. For regular random variables, the area below and above the curve should be approximately equal.
\newline
\paragraph{\textit{The Hoeffding Bootstrap, Continued}} With $\mathcal{U}$ random variables introduced, we now provide an informal justification for extending a slightly more conservative version of the Hoeffding bootstrap to an arbitrary $T_0$. This justification makes use of the principle of indifference, which states that the assignment of a uniform measure minimizes risk in the absence of information. Although it is preferable to commence from 'known' statements to derive a bound on uncertainty, we show that employing this principle is consistent with bounds derived under oracle assumptions in all circumstances except the most extreme. Importantly, for this exploration, we do condition on the empirical distribution.

We introduce some notation first. Say $T_*$ is a bootstrapped statistic s.t. $T_* \sim F_{T_*}(t)$, where $F_{T_*}(t) = \Psi \{t, \hat{F}_n (y) \}$ and $F_{T_0} = \Psi \{t, F(y) \}$. The $\Psi$ notation indicates that $F_{T_0}$ is a functional of the marginal marginal distribution function(s). Importantly, this does not assume independence since $\Psi \{t, F(y) \}$ can be complicated in an unknown fashion as a consequence of probabilistic dependencies. Like before, we then say $\{ T_{*, k} \}_{k \in \mathcal{K}}$ is a sample of bootstrap statistics identically and independently drawn from $F_{T_*}$, except now $\mathcal{K} = \{1, 2, \ldots, B  \}$. Here, we only assert 1) that there is some $N$ s.t. for all $n > N$, it is true that $T_0 \in \mathcal{U}$ and 2) that that the minimum and maximum values of the support of $F_{T_*}$, say $\{m_{*, 0}, M_{*, 0} \}$, are finite almost surely: a fact that is already implied by bounded nature of $T_0$. These presuppositions are relatively light. Only 2) is truly necessary.

Since $F_{T_*}(t)$ is a random CDF, it is also safe to assert that $\{m_{*, 0}, M_{*, 0} \}$ are random variables. For functions $q_m$ and $q_M$ then, say $m_{*, 0} = m_0 + q_m (Y_1, Y_2, \ldots, Y_n)$ WLOG. Constructing this object is always valid since $q_m = m_{*, 0} - m_0$ is defined on the same probability space as $m_{*, 0}$. Symmetry of argument supplies the same equation for $M_{*, 0}$: $M_{*, 0} = M_0 + q_M(Y_1, Y_2, \ldots, Y_n)$. Consequently:
$$M_{*, 0} - m_{*, 0} = M_0 - m_0 + \{ q_M(Y_1, Y_2, \ldots, Y_n) - q_m(Y_1, Y_2, \ldots, Y_n) \}$$
Now, define $q_{M-m} = q_M(Y_1, Y_2, \ldots, Y_n) - q_m(Y_1, Y_2, \ldots, Y_n)$ for brevity. It is apparent that when $q_{M-m} \geq 0$, it follows that $M_{0} - m_{0}  \leq M_{*, 0} - m_{*, 0}$. Because $q_{M-m} \geq 0$ provides us with a desired bound, diligence dictates examining the opposite valence.

Due to the fact that $Y$ is non-degenerate by tacit assumption, we can also assert with confidence that $-(M_0 - m_0) < q_{M-m}$. Thereby, reasoning conservatively requires us to stipulate a lower bound $LB_{q_{M-m}}$ s.t. $-(M_0 - m_0) < LB_{q_{M-m}} < 0$. Since every irrational number $c \in (0, 1)$ can be approximated arbitrarily well by a rational number, we can define two unknown integers $a$ and $b$ s.t. $a < b$ to conclude that $M_0 - m_0 \leq \{b - a \}^{-1} \cdot b \cdot \{ M_{*, 0} - m_{*, 0} \}$ almost surely. As a consequence, for any bounded statistic $T_0$ and arbitrary $\epsilon > 0$, there exists integers $a < b$ s.t. $\text{Pr}(|T_0 - \text{E}T_0| > \epsilon) \leq 2 \cdot \text{exp} \{- \{b \cdot (M_{*, 0} - m_{*, 0}  ) \}^{-2} \cdot \{ b-a \}^2 \cdot 2\epsilon^2  \}$. Call this the bootstrap concentration inequality.

By the form of the inequality, it is apparent that $\{a, b \}$ are related to the dependency structure and that $a$ especially controls extreme behavior. For instance, as $a \to b$, the bootstrap bound becomes trivial. If $b \to \infty$ much faster than $a$, then we recover the more efficient bound. Further modeling work that relates $\{a, b \}$ to the dependency structure of $\{Y_i \}_{i \in I_n}$ will thus be fortuitous. Just the same, additional discussion that incorporates a wider array of prior distributional assumptions on $\{a, b \}$ will indubitably be interesting and beneficial.

Out of ignorance, for our purposes, we cite the principle of indifference, which places a uniform measure on $[-(M_0 - m_0), 0]$. This suggests that $LB_{q_{M-m}} \equiv -2^{-1} \cdot (M_0 -m_0)$ is a defensible choice that minimizes risk since it is the expected value of the (constrained) $q_{M-m}$ function. Doing so is also equivalent to setting $b \equiv 2$ and $a \equiv 1$. As a consequence, without further information, it is implied that $M_0 - m_0 \leq 2 \cdot \{ M_{*, 0} - m_{*, 0} \}$. An employment of Hoeffding's inequality then supplies that $T_0 \pm 2 \cdot \{  M_{*, 0} - m_{*, 0} \} \cdot \sqrt{2^{-1} \text{log}(2/\alpha)}$ is a confidence set for $\alpha \in (0, 1)$ with at least $1-\alpha$ coverage. If $T_0 \in \mathcal{U}$ for sufficiently large $n$, then $T_0 \pm 2 \cdot \{  M_{*, 0} - m_{*, 0} \} \cdot \sqrt{6^{-1} \text{log}(2/\alpha)}$ can be used as an improvement. Still, this approach is more conservative than the one in Section 2. Altogether, the improved bound on the tail probabilities for this section inflates the error around the estimate by an approximate factor of 1.15 in comparison.

$ M_{*, 0}$ and $m_{*, 0}$ are also unknown. However, since $T_{*, (1)} \overset{a.s.}{\to} m_{*, 0}$ for all $n$ as $B \to \infty$ WLOG and $\{ T_{*, k}\}_{k \in \mathcal{K}}$ can be made arbitrarily large, $M_{*, 0} - m_{*, 0}$ can be replaced with $T_{*, (K)} - T_{*, (1)}$ as a plug-in with negligible error for large enough $B$. If the empirical distribution of $T_{*, k}$ demonstrates heavy tails, then typical choices for $B$ will likely suffice. However, if $T_{*, k}$ possesses a light tail, $B$ will need to be much larger to compensate for the sub-optimal convergence rate of extreme order statistics.
\newline
\paragraph{\textit{A Defense of the Principle of Indifference}} From here, we substantiate the use of the rule of indifference with a supplementary exploration. We show that the use of an oracle assumption leads to the same bounds in all but the most extreme of situations. The oracle assumption is as follows: for large enough $n$, $\text{E}T_0$ is contained in $\mathcal{S}_{T_*}$ in the sense that $m_{*, 0} \leq \text{E}T_0 \leq M_{*, 0}$. We say that this is an oracle assumption because it automatically supplies a population bound of probability one on the expected value. 

As a caveat, note that using this assertion is in no way illegal statistically if it can be justified. Taking a wider view, there is no thematic difference between this strategy and the supposition of a particular dependency structure so that asymptotic normality can be inferred. In place of estimating variance parameters, we would instead be replacing $M_{*, 0}$, say, with its bootstrap estimator. Nevertheless, taking a route that does not suppose these extra conditions directly is preferable. For instance, if C5' holds and the bootstrap \textit{minimally} works in the sense that if $T_0 \overset{a.s.}{\to} \text{E}T_0$ then $\text{E}T_* \to \text{E}T_0$ as $n \to \infty$, the oracle condition is relatively safe to assume in many contexts, at least for moderate sample sizes. Plug-ins for estimands that are linear functionals of $F(y)$ will often qualify, as will unbiased statistics more generally.

Now, we commence the defense. When comparing the extremes of $\mathcal{S}_{T_0}$ and $\mathcal{S}_{T_*}$, there are only four possible cases:
\begin{enumerate}
\item[(1)] $m_{*, 0} \leq m_0$ and $M_0 \leq M_{*, 0}$
\item[(2)] $m_{*, 0} \leq m_0$ and $M_{*, 0} \leq M_0$
\item[(3)] $m_0 \leq m_{*, 0}$ and $M_0 \leq M_{*, 0}$
\item[(4)] $m_0 \leq m_{*, 0}$ and $M_{*, 0} \leq M_{0}$
\end{enumerate}
Case (1) is trivial and automatically implies that $M_0 - m_0 \leq M_{*, 0} - m_{*, 0} \leq 2\cdot \{ M_{*, 0} - m_{*, 0} \}$. One can generally expect case (1) to hold under conditions of mutual negative dependence. Cases (2) and (3) provide the targeted bound in conjunction with the oracle statement. Using it allows us to conclude that $m_{*, 0} - m_0 \leq 0 \leq M_{*, 0} - \text{E}T_0$ for case (2) at no loss of generality. As a consequence, if $T_0 \in \mathcal{U}$, it then follows that $M_0 - m_0 \leq 2 \cdot \{ M_{*, 0} - m_{*, 0} \}$.

Cases (2) and (3) are likely to hold when dependencies are predominantly negative or positive, but they are not extreme in magnitude and the average number of dependencies is not linear in $n$. It is only case (4) that is problematic. For this case, $m_{*, 0}-m_0$ is strictly positive, for instance. Asserting that $m_{*, 0}-m_0 \leq M_{*, 0} - \text{E}T_0$ is thus even more non-trivial. Situations of extreme dependence s.t. a large proportion of the outcome variables are positively dependent in a strong and redundant sense can induce this case. In an extreme example, consider  sampling $\{ Y_i \}_{i \in I_n}$ but each outcome variable is secretly equal to $Y_1$ with additional random noise. Then $T_0$ is 'almost' a function of just one variable and the support of this statistic will be more similar to the support of $T(Y_1)$. In other words, it will not concentrate. C5' prevents a large proportion of these extreme cases by definition, and also because it is sufficient for establishing the uniform and almost sure convergence of the empirical CDF.

Nevertheless, to be comprehensive, we show that the desired bound holds for this case if and only if the principle of indifference holds. We can show one sub-case WLOG under the supposition of regular $\mathcal{U}$ status since $\int_{\mathcal{T_0}} F_{T_0} dt = \int_{\mathcal{T_0}} \{1- F_{T_0} \} dt $. This exercise applies to discrete random variables since integrating CDFs that are step-functions over an interval with the same endpoints as the support still supplies the targeted values. Now, assume $m_{*, 0}-m_0 \leq M_{*, 0} - \text{E}T_0$ at no loss of generality. If $T_0 \in \mathcal{U}$, then this ensures the bound. Observe the following, however, where $\implies$ should be read as 'implies':
\begin{align*}
m_{*, 0}-m_0 & \leq M_{*, 0} - \text{E}T_0 \implies \\
q_m & \leq M_0 + q_M -  \text{E}T_0 \implies \\
-q_{M-m} & \leq M_0 -  \text{E}T_0 \implies \\
-q_{M-m} & \leq 2^{-1} \{M_0 - m_0 \}
\end{align*}
The left hand side of the inequality is maximized when $q_{M-m} = -2^{-1}\{M_0 - m_0 \}$, which gives us the same assertion as the principle of indifference. Moving in the other direction entails a reversal of the logic and is omitted.

In conclusion, under the mild assumption of $\mathcal{U}$ limiting behavior, three of the four exhaustive cases supply the same bound without using the principle indifference as a premise. Therefore, statisticians can feel comfortable asserting it insofar as the most extreme of dependency scenarios are reasonably excluded. Since many popular statistics conceivably adopt $\mathcal{U}$-like behavior for even moderate sample sizes, this is a fecund route for inference in the all-too-common face of unknown and intractable systems of statistical dependencies. Establishing this route more rigorously will undoubtedly be a boon. Finally, we note that the analysis of this section can be generalized to the situation s.t. we do not condition on the empirical distribution. In this scenario, we already know that $\text{E}T_0 \in [m, M]$. However, if the dependency conditions are violated, we cannot consistently estimate these extremes. In this case, we can replace $M_{*, 0} - m_{*, 0}$ with $\hat{M} - \hat{m}$ in the statements above and still employ the same arguments with little alteration to conclude that the same, slightly more conservative confidence set is a defensibly principled choice, and even when $\hat{M} - \hat{m}$ is statistically inconsistent.

\subsection{Counterfactual Linear Regression}
Linear regression is a popular tool for causal and predictive inference. For the former, inverse probability weighting of the marginal structural model or standardization of the adjusted model are common approaches \cite{hernan2010causal, hernan2006estimating, imbens2004nonparametric, mansournia2016inverse}. The marginal structural model w.r.t. an event $\{ T=t \}$ is defined as follows:
\begin{equation}
\text{E}Y^t = \beta_0 + \beta_1 t
\end{equation}
The adjusted model is defined similarly in conjunction with a vector of adjusting variables $\mathbf{L}$, which are posited to conditionally fulfill C2:
\begin{equation}
\text{E}(Y|t, \mathbf{L})= \text{E}(Y^{t}| \mathbf{L}) = \beta_{0_*} + \beta_1 t + \mathbf{L}\boldsymbol{\beta}
\end{equation}
Standardization then yields that $\text{E}\{ \text{E}(Y^{t}| \mathbf{L}) \} = \beta_{0_*} + \beta_1 t + \text{E} \{ \mathbf{L} \} \boldsymbol{\beta} = \text{E}Y^t$ under C1-C3. As aforementioned, the identification of a vector $\mathbf{L}$ that achieves C2 for all levels of $T$ that are required for causal contrasts is non-trivial. This sentiment generalizes to the identification of a vector of random variables that are sufficient for estimating the probability model that is necessary for the inverse probability of treatment weighting. The difficulty is further exacerbated in both cases by informative sampling. For instance, for eq. (3.2), what is actually estimated is: $\text{E}_{\delta}(Y|t, \mathbf{L}) = \alpha_0 + \alpha_1 t + \mathbf{L}_{\delta}\boldsymbol{\alpha}$. Hence, even when C2 is met, only $\text{E}_{\delta}Y^t$ is identifiable in the absence of non-informative sampling or additional constraints. $\text{E}_{\delta}Y^t$, however, might not be the target of interest. 

We now use concepts from the previous sections to demonstrate that the core assumptions of linear regression for predictive (associative) inference are sufficient for the identification of causal parameters in conjunction with C4. To this end, we specify the data-generating mechanism for the linear model in eq. (3.3) below with $\mathbf{L=l}$ fixed. 
\begin{equation}
Y_{\delta_i} = \alpha_0 + \alpha_1 t_i + \mathbf{l}_i \boldsymbol{\alpha} + \epsilon_{\delta_i}
\end{equation}
Traditionally, eq. (3.3) requires that $\text{E}_{\delta}(\epsilon_i | t_i, \mathbf{l}_i) =0$ for $\forall i$ if only predictive inference is the goal and all covariate patterns of interest have truly been experimentally fixed. This is weaker than strict exogeneity, which requires that $\text{E}_{\delta}(\epsilon_i | T_i, \mathbf{L}_i) =0$ when $T_i$ and $ \mathbf{L}_i$ are stochastic. Using standardization requires a slightly weaker form of strict exogeneity, conditioned on all treatment contrasts of interest, since the method averages over $\mathbf{L}_{\delta}$: $\text{E}_{\delta}(\epsilon_i | t_i, \mathbf{L}_i) =0$ for $\forall i$. Otherwise, for finite sample inference, the second core assumption is that $\epsilon_{\delta_i} \sim N(0, \sigma^2_i)$ for $\forall i$.

The first core assumption is commonly evaluated using the predicted versus residual plot. Under the working proposition of valid specification, this plot should demonstrate an approximately symmetric scatter of the residuals about the horizontal zero line for any arbitrarily small neighborhood around any predicted point on the x-axis.

To make use of these traditional conditions, we must first make an inconsequential adjustment to the assumption of normality. We are working within a universe of bounded random variables. Consequently, the $\epsilon_{\delta_i}$ of eq. (3.3) cannot be normally distributed. This is no great loss for four related reasons. Firstly, in a grand majority of scientific investigations, $Y_{i}$ is bounded. For example, if each $Y_i$ is a measurement of a person's blood pressure, it is impossible for it to be less than zero or greater than an arbitrary real number. Its distribution cannot be supported on a set that is equal to $\mathbb{R}$. This automatically implies that the $\epsilon_i$ cannot truly be normally distributed. In these situations, when statisticians assume normality, it is intended as a feasible approximation that results in negligible error, and that is fecund mathematically.

The second reason is similar to the first. Even if someone wishes to insist that $Y_i$ is supported on the entire real line, $Y_{\delta_i}$ often cannot be due to the intrinsic limitations of measurement and observation. Thirdly, as hinted in the first reason, the assumption of normality can be replaced with the assumption that $\epsilon_{\delta_i}$ has a CDF $F_i(e_{\delta})$ of a normal distribution that has been symmetrically truncated around zero. This is equivalent to positing that each $\epsilon_{\delta_i}$ is related some variable $Z_i \sim N(0, \sigma^2_{Z_i})$ s.t. for an (almost) arbitrarily small $\tau > 0$, $\text{Pr}(-M_i \leq Z_i \leq M_i ) = 1 - \tau$ and $F_i(e_{\delta}) = (1-\tau)^{-1} \Phi_{Z_i}(e_{\delta})$ on $[-M_i, M_i]$. Provided this setup, the bias that results from treating $\epsilon_{\delta_i}$ as strictly normally distributed for mathematical convenience is unimportant, especially since one does not need to identify $\sigma^2_{Z_i}$.

The fourth and last reason, which motivates the next proposition, is related to requirement of symmetric scatter in the residual versus fitted plot. A symmetrically truncated normal distribution is a special case of a $\mathcal{U}$ random variable. Moreover, when a continuous random variable has an expected value of zero, all that is required for regular $\mathcal{U}$ status is for it to be supported on a symmetric interval $[-M_i, M_i]$.

Hence, the typical set of assumptions already employed for fixed linear regression already requires that each $\epsilon_{\delta_i} \in \mathcal{U}$. Additionally, we also note that positing only that $\epsilon_{\delta_i} \in \mathcal{U}$ $\forall i$ is a fundamentally weaker assumption than (symmetrically truncated) normality. Under this milder condition, a researcher only needs to verify that the residual versus predicted plot is (approximately) symmetrically \textit{supported} around zero about any neighborhood of predicted values. The behavior of the scatter within any neighborhood is otherwise unimportant, insofar as it reasonably justifies that the expected value is also zero. Nevertheless, it is apropos to state that, if only this milder condition is supposed, then the concentration inequality of \textbf{Table 1} or a central limit theorem are required for the construction of confidence intervals. Of course, under copious amounts of dependencies, a central limit will not necessarily apply.

We now present a useful main result in Proposition 5, although it is technically a more detailed case of Proposition 3. The extra assumption of regularity is not strictly necessary.
\begin{proposition}
Assume C1 and C4. Say $Y_{t, \mathbf{l}, \delta} = g(t, \mathbf{l}, \boldsymbol{\alpha}) + \epsilon_{t, \mathbf{l}, \delta}$ for some measurable (possibly monotonic) function $g$. Suppose each $\epsilon_{t, \mathbf{l}, \delta}$ is regular, $\text{E}_{\delta}(\epsilon | t, \mathbf{l}) = 0$, and $\epsilon_{t, \mathbf{l}, \delta} \in \mathcal{U}$ for $\forall t, \mathbf{l}$ fixed. Then $\text{E}_{\delta}Y_{t, \mathbf{l}} = \text{Av}(Y_{\mathbf{l}}^{t})$.
\end{proposition}
\begin{proof}
Let $t, \mathbf{l}$ be arbitrary. Then $\text{E}_{\delta}Y_{t, \mathbf{l}} = g(t, \mathbf{l}, \boldsymbol{\beta})$ since $\text{E}_{\delta}(\epsilon | t, \mathbf{l}) = 0$.

For an arbitrary bounded random variable $Z$, say $\min (Z) = \underset{z \in \mathcal{S}_Z}{\min}(z)$ and $\max (Z) =  \underset{z \in \mathcal{S}_Z}{\max}(z)$, the greatest lower and least upper bounds of the closed set $\mathcal{S}_Z$, respectively. Since $g(t, \mathbf{l}, \boldsymbol{\beta})$ is a constant, it follows that $\min (Y_{t, \mathbf{l}, \delta}) =  g(t, \mathbf{l}, \boldsymbol{\beta}) + \min (\epsilon_{t, \mathbf{l}, \delta})$ and $\max (Y_{t, \mathbf{l}, \delta}) =  g(t, \mathbf{l}, \boldsymbol{\beta}) + \max (\epsilon_{t, \mathbf{l}, \delta})$. Moreover, since each $\epsilon_{t, \mathbf{l}, \delta}$ is regular, then each $Y_{t, \mathbf{l}, \delta}$ is also obviously regular.
From here:
\begin{align*}
\min (Y_{t, \mathbf{l}, \delta}) + \max (Y_{t, \mathbf{l}, \delta}) & = 2 g(t, \mathbf{l}, \boldsymbol{\beta}) + \min (\epsilon_{t, \mathbf{l}, \delta}) + \max (\epsilon_{t, \mathbf{l}, \delta}) \implies \\
2^{-1} \{ \min (Y_{t, \mathbf{l}, \delta}) + \max (Y_{t, \mathbf{l}, \delta}) \} & = g(t, \mathbf{l}, \boldsymbol{\beta}) + 2^{-1} \{ \min (\epsilon_{t, \mathbf{l}, \delta}) + \max (\epsilon_{t, \mathbf{l}, \delta}) \} \implies \\
\text{Av}(Y_{t, \mathbf{l}, \delta}) & = g(t, \mathbf{l}, \boldsymbol{\beta}) + 0 \implies \\
\text{E}_{\delta}Y_{t, \mathbf{l}} & = \text{Av}(Y_{t, \mathbf{l}, \delta}) =  \text{Av}(Y_{\mathbf{l}}^{t})
\end{align*}
The third line follows from regularity and the fact that $\epsilon_{t, \mathbf{l}, \delta} \in \mathcal{U}$ for $\forall t, \mathbf{l}$ fixed. The last line follows from substitution, C1, and C4.
\end{proof}
Set $g\{ t, \mathbf{l}, (\alpha_0, \alpha_1, \boldsymbol{\alpha}) \} = \alpha_0 + \alpha_1 t + \mathbf{l} \boldsymbol{\alpha}$ to recover eq. (3.3). Under the assumption that $(t', \mathbf{l})$ and $(t'', \mathbf{l})$ have both been fixed, where the values $t'$ and $t''$ represent the treatment values to be contrasted, Proposition 5 implies that $\alpha_1 \propto  \text{Av}(Y_{\mathbf{l}}^{t'}) -  \text{Av}(Y_{\mathbf{l}}^{t''})$: the difference in the \textit{average} values of $Y_{\mathbf{l}}^{T}$ when $T=t'$ and $T=t''$. When $(t', \mathbf{l})$ and $(t'', \mathbf{l})$ have not been fixed but at least all treatment values have been, i.e., when the researcher did not fix $\mathbf{l}$ for all contrasts of scientific interest, the previous statement still holds in general when $\text{E}_{\delta}(\epsilon | t, \mathbf{L}) = 0$ or $\forall t$ involved. The proof of Proposition 5 would only need to use this stronger statement with no further change. If the researcher wishes to reason about contrasts of $T$ that have not been fixed as well, then the (non-trivial) assumption that $\text{E}_{\delta}(\epsilon | T, \mathbf{L}) = 0$ suffices.

This proves that conditions that are weaker than those traditionally supposed for making inferences about associations are sufficient for inference w.r.t. causal parameters. A statistician can target these parameters using linear regression under (almost) arbitrary sampling bias, insofar as linearity and $\mathcal{U}$ status in the error distributions are feasibly defensible w.r.t. the sample measures and at least the supports have been preserved. Even this last condition can be weakened further since we truly only require the preservation of functional averages.

We choose to highlight three additional points of interest. The first one is about the interpretation of $\alpha_1$. It uses similar language to current interpretations. However, care is due. Although the word 'average' is often used for interpreting the coefficients of typical linear regressions, this is imprecise slang for the change in expected value. The word 'average' is also imprecise for the functional average, but it is at least closer in spirit in comparison to other applications since the functional average is a uniform averaging of the support.

The second point is that Proposition 5 enables reasoning about counterfactual conditional functional averages in high-dimensional settings. Often, however, the researcher cares mostly about a marginal estimate. Although they can serve a similar purpose, $\text{E} \{ \text{Av}(Y_{\mathbf{L}}^t) \} \neq \text{Av}(Y^t)$ and $\text{Av} \{ \text{Av}(Y_{\mathbf{L}}^t) \} \neq \text{Av}(Y^t)$ in general. This does not signify that these parameters do not possess scientific meaning. For example, $\text{E} \{ \text{Av}(Y_{\mathbf{L}}^{t'}) \} - \text{E} \{ \text{Av}(Y_{\mathbf{L}}^{t''}) \}$ can still be interpreted as an \textit{expected} functional average effect over $\mathbf{L}$.

Nevertheless, one special circumstance when the identity $\text{E} \{ \text{Av}(Y_{\mathbf{L}}^t) \} = \text{E}Y^t$ does hold is when C2, the conditions of Proposition 5 under stricter exogeneity, and the additional requirement that $Y^t \in \mathcal{U}$ are valid.

The last point is concise to state. The analysis of variance (ANOVA) is a special case of the linear regression model presented. It is an elementary tool that is ubiquitous in research. All prior discussion and Proposition 5 therefore apply to ANOVA procedures under conditions that are already stipulated. Hence, insofar as C1 and C4 are defensible, this means that a plethora of prior work can be re-interpreted with a restricted causal lens in partnership with a structural causal model.
\section{Monte Carlo Simulations}
Before we apply our strategy to real data, we illustrate its utility with a set of Monte Carlo simulations that show how functional average and $\mathcal{U}$ concepts are useful for causal inference. For simplicity, we proceed with non-informative sampling conditions that presuppose mutual independence. All simulations use $M=1,000$ iterations for sample sizes $n \in \{500, 2500, 5000, 10000 \}$. Furthermore, all constructed $1-\alpha$ confidence sets use $\alpha = .05$. Three main simulation experiments are provided in total. The first examines the behavior of basic functional average estimators for symmetric and non-symmetric distributions. The second and third simulations demonstrate that causal effects can be consistently estimated without controlling for confounding. All experiments examine the performance of Hoeffding style bootstrapping procedures.
\subsection{Univariate Functional Average Estimation}
The first simulation of this experiment examines the performance of $\hat{MR}$ for three truncated normal distributions: $Y_1 \sim TN(m=0, M=20, \mu = 10, \sigma = 5), Y_2 \sim TN(m=0, M=15, \mu = 10, \sigma = 3)$, and $Y_3 \sim TN(m=0, M=15, \mu = 5, \sigma = 3)$. The first random variable $Y_1$ is in the $\mathcal{U}$ class and hence $\theta_1 = \text{Av}(Y_1) = \text{E}Y_1 = 10$. However, $Y_1$ and $Y_2$ are not $\mathcal{U}$ variables. Their distributions are skewed with tails that impact convergence behavior. For these variables, $\theta = \text{Av}(Y) = 7.5$.

Hoeffding bootstrap style confidence sets $(\hat{CI}_H)$ are constructed as described in Section 2. To contrast its performance, we also use an $m$-out-of-$n$ bootstrap. Again, an important requirement of the $m$-out-of-$n$ bootstrap is that $m \to \infty$, but $n^{-1}m \to 0$. To meet this criteria, we set $m = r(\sqrt{n})$ since this setting produces relatively conservative results. Hence, if it fails to perform well, this highlights the utility of the Hoeffding procedure. Percentile confidence sets $(\hat{CI}_{p, m})$ are constructed from the $m$-out-of-$n$ bootstraps. For reference we also construct Hoeffding style confidence sets $(\hat{CI}_{H, m})$ from this process. Importantly, all bootstrap procedures makes use of simple random sampling with replacement and only $500$ bootstrap samples. Although sub-optimal, a low number of bootstrap samples is used to limit computational burden. A decent performance of the Hoeffding bootstrap at $B=500$ is still a good indicator. Empirical coverage rates are estimated with $EC = M^{-1} \sum_{i=1}^M 1_{\theta \in \hat{CI}_i}$ WLOG.

\textbf{Table 2} presents the results of this experiment. Importantly, all values in tables henceforth represent the arithmetic average of simulated objects, including the endpoints of confidence sets.
\begin{table}[H]
\caption{Functional Average Estimation, Continuous}
\centering
\begin{adjustbox}{width=1\textwidth}
\begin{tabular}{ccc|cc|cc|cc}
  \hline
     \\[-3.5\smallskipamount]
$\text{Av}(Y)$ & n & $\hat{MR}$ & $\hat{CI}_{H, m}$ & $EC_{H, m}$ & $\hat{CI}_{p, m}$ & $EC_{p, m}$ & $\hat{CI}_{H}$ & $EC_{H}$ \\ 
   \\[-3.5\smallskipamount]
  \hline
     \\[-3.5\smallskipamount]
$\theta = 10$ & 500 & 10.01 & (1.85, 18.15) & 1 & (8.02, 11.98) & 1 & (8.87, 11.13) & 1 \\ 
   & 2500 & 10 & (4.5, 15.5) & --- & (8.72, 11.27) & --- & (9.72, 10.28) & --- \\ 
   & 5000 & 10 & (5.42, 14.58) & --- & (8.97, 11.03) & --- & (9.86, 10.14) & --- \\ 
   & 10000 & 10 & (6.26, 13.74) & --- & (9.17, 10.83) & --- & (9.93, 10.07) & --- \\
      \\[-3.5\smallskipamount]
   \hline 
      \\[-3.5\smallskipamount]
 $\theta = 7.5$  & 500 & 8.11 & (2.1, 14.12) & 1 & (7.58, 10.64) & 0.40 & (6.51, 9.7) & 0.88 \\ 
   & 2500 & 7.74 & (2.98, 12.5) & --- & (7.6, 10.09) & 0.30 & (6.82, 8.65) & 0.94 \\ 
   & 5000 & 7.63 & (3.31, 11.96) & --- & (7.59, 9.87) & 0.25 & (6.96, 8.31) & 0.97 \\ 
   & 10000 & 7.58 & (3.7, 11.46) & --- & (7.58, 9.67) & 0.21 & (7.13, 8.03) & 0.98 \\ 
      \\[-3.5\smallskipamount]
   \hline
      \\[-3.5\smallskipamount]
 $\theta = 7.5$  & 500 & 6.91 & (0.88, 12.94) & 1 & (4.37, 7.44) & 0.43 & (5.31, 8.51) & 0.90 \\ 
   & 2500 & 7.26 & (2.46, 12.06) & --- & (4.91, 7.4) & 0.30 & (6.34, 8.18) & 0.95 \\ 
   & 5000 & 7.37 & (3.04, 11.69) & --- & (5.13, 7.41) & 0.23 & (6.69, 8.04) & 0.97 \\ 
   & 10000 & 7.42 & (3.52, 11.33) & --- & (5.33, 7.42) & 0.18 & (6.97, 7.87) & 0.97 \\ 
      \\[-3.5\smallskipamount]
   \hline
\end{tabular}
\end{adjustbox}
\caption*{\small $\dagger$ $\theta_i$ corresponds to the functional average of the $Y_i$ TN distributions introduced above; '---' indicates that the value is the same as the first row}
\end{table}
These results demonstrate that $\hat{MR}$ behaves as intended. It is unbiased for the symmetric distribution and convergence behavior---albeit slow---is observable w.r.t. the target parameters for the asymmetric distributions that exhibit more problematic tail behavior. Importantly, the Hoeffding-style bootstrap also appears to behave as intended. Although it failed to uphold nominal coverage values for the skewed distributions at $n=500$, it quickly overcame this behavior to provide conservative empirical coverage for $\theta$. Also, the difficulties of efficiently estimating $\text{Av}(Y)$ are evident for non-$\mathcal{U}$ variables, highlighting the utility of this type of variable. Notably, the $m$-out-of-$n$ percentile interval did not perform well even when $m$ was $O(\sqrt{n})$.

The second experiment is for discrete variables. Like before, we use three different truncated normal distributions: $Y_1 \sim TN(m=0, M=40, \mu = 20, \sigma = 5), Y_2 \sim TN(m=0, M=40, \mu = 25, \sigma = 8)$, and $Y_3 \sim TN(m=0, M=40, \mu = 15, \sigma = 8)$. These random variables are uniformly rounded to the nearest integer to induce discreteness. Here, we contrast $\hat{MR}$ with $\hat{Av}$, the discrete plug-in for eq. (2.1). Hoeffding style bootstraps are again employed to construct confidence sets. However, now we use a $\mathcal{U}$ approximation in accordance with Section 3. Even if $\mathcal{U}$ status does not hold exactly, insofar as convergence behavior to a constant holds, the method should remain robust.

Results for this simulation are available in \textbf{Table 3}. We no longer examine the performance of the $m$-out of-$n$ bootstrap.
\begin{table}[H]
\caption{Functional Average Estimation, Discrete}
\centering
\begin{adjustbox}{width=1\textwidth}
\begin{tabular}{cc|ccc|ccc}
  \hline
    \\[-3.5\smallskipamount]
$\text{Av}(Y) = 20$ & n & $\hat{Av}$ & $\hat{CI}_{H}$ & $EC_{H}$ & $\hat{MR}$ & $\hat{CI}_{H}$ & $EC_{H}$ \\ 
  \\[-3.5\smallskipamount]
  \hline
    \\[-3.5\smallskipamount]
$r(Y_1)$ & 500 & 20.012 & (15.21, 24.82) & 1 & 20.044 & (13.88, 26.21) & 1 \\ 
   & 2500 & 19.992 & (15.84, 24.14) & --- & 19.969 & (14.93, 25.01) & --- \\ 
   & 5000 & 20.014 & (16.18, 23.85) & --- & 20.023 & (15.58, 24.47) & --- \\ 
   & 10000 & 20.020 & (16.39, 23.65) & --- & 20.016 & (16, 24.04) \\ 
     \\[-3.5\smallskipamount]
   \hline
     \\[-3.5\smallskipamount]
 $r(Y_2)$  & 500 & 21.921 & (17.62, 26.22) & 0.971 & 21.218 & (16.48, 25.96) & 0.955 \\ 
   & 2500 & 20.490 & (18.19, 22.79) & 0.980 & 20.365 & (18.17, 22.56) & 0.963 \\ 
   & 5000 & 20.194 & (18.55, 21.83) & 0.985 & 20.169 & (18.67, 21.67) & 0.975 \\ 
   & 10000 & 20.051 & (19, 21.1) & 0.972 & 20.050 & (19.06, 21.04) & 0.957  \\ 
     \\[-3.5\smallskipamount]
   \hline
     \\[-3.5\smallskipamount]
 $r(Y_3)$  & 500 & 18.053 & (13.8, 22.31) & 0.963 & 18.788 & (14.04, 23.54) & 0.953 \\ 
   & 2500 & 19.518 & (17.22, 21.82) & 0.985 & 19.638 & (17.45, 21.83) & 0.967 \\ 
   & 5000 & 19.803 & (18.16, 21.44) & 0.987 & 19.833 & (18.34, 21.32) & 0.980 \\ 
   & 10000 & 19.941 & (18.87, 21.01) & 0.977 & 19.942 & (18.94, 20.94) & 0.963 \\ 
     \\[-3.5\smallskipamount]
   \hline
\end{tabular}
\end{adjustbox}
\end{table}
The results of the discrete simulation largely match those of the preceding continuous one. The plug-in estimator performed more favorably than the sample mid-range only for $r(Y_1)$: the distribution with the lightest tails. Importantly, although many simulations demonstrated departures from $\mathcal{U}$ status, the coverage remained robust due to the conservative nature of the Hoeffding bootstrap. Further details pertaining to this fact are available in the supplementary material.
\subsection{Causal Inference with Functional Averages}
Next, we demonstrate a classical case of confounding where the functional average treatment effect is identifiable and statistically consistently estimable without any adjustment. We use the following variables for this simulation: $C \sim Bern(.5)$, $T \sim Bern(.3 + .5C)$, and $e \sim TN(m=-50, M=50, \mu = 0, \sigma = \tau)$. Moreover, we will say that $Y^{t=1} = 110 + 50C + e$, $ Y^{t=0} = 100 + 50C + e$, and $Y = Y^{t=1}T + (1-T)Y^{t=0}$.

The confounding variable is $C$, which is present in half of the theoretical population on average. When $C=1$, the probability of allocation to treatment is larger. The structure of the confounding also preserves the support of the counterfactual distribution w.r.t. the observed one. We also set $\tau$ to $5$ then $25$ to demonstrate the performance of a functional average estimator when the tail probabilities are thin or heavy.

We contrast simple linear regression with $\hat{MR}$ for estimating $\Delta = \text{Av}(Y^{t=1}) -  \text{Av}(Y^{t=0}) = 10$. Here, the mid-range estimator of $\Delta$ is $\hat{\Delta}_{MR} = 2^{-1}\{ Y_{(1), t=1} + Y_{(n_1), t=1} \} - 2^{-1} \{ Y_{(1), t=0} + Y_{(n_0), t=0} \}$. This simulation is executed WLOG since it can be implicitly assumed that the researcher has stratified on some set of variables---such as propensity scores---to limit confounding bias or achieve the equality of supports. Confidence sets and estimators of the empirical coverage are all otherwise constructed as previously explored using the Hoeffing bootstraps of Section 2. Power is estimated by $EP_H = M^{-1} \sum_{i=1}^M 1_{0 \notin \hat{CI}_{H_i}}$. \textbf{Table 4} has the results.
\begin{table}[H]
\caption{Continuous Effect Estimators: $\Delta = 10$}
\centering
\begin{adjustbox}{width=1\textwidth}
\begin{tabular}{ccc|ccc}
\hline
   \\[-3.5\smallskipamount]
$\tau = 5$ & n & $\hat{\Delta}_{OLS}$ & $\hat{\Delta}_{MR}$ & $\hat{CI}_H$ & $EP_H$ \\ 
   \\[-3.5\smallskipamount]
  \hline
     \\[-3.5\smallskipamount]
 & 500 & 33.45 & 11.76 & (1.53, 22) & 0.71 \\ 
   & 2500 & 34.39 & 11.69 & (3.05, 20.33) & 0.85 \\ 
   & 5000 & 35.09 & 11.65 & (3.67, 19.63) & 0.90 \\ 
   & 10000 & 35.18 & 11.47 & (3.88, 19.07) & 0.92 \\ 
      \\[-3.5\smallskipamount]
   \hline
      \\[-3.5\smallskipamount]
   $\tau = 25$ & n & $\hat{\Delta}_{OLS}$ & $\hat{\Delta}_{MR}$ & $\hat{CI}_H$ & $EP_H$ \\ 
      \\[-3.5\smallskipamount]
     \hline
        \\[-3.5\smallskipamount]
 & 500 & 33.44 & 12.82 & (-10.47, 36.11) & 0.05 \\ 
   & 2500 & 34.40 & 10.84 & (2.73, 18.95) & 0.91 \\ 
   & 5000 & 35.11 & 10.47 & (5.79, 15.15) & 1 \\ 
   & 10000 & 35.20 & 10.25 & (7.68, 12.82) & --- \\ 
      \\[-3.5\smallskipamount]
   \hline
\end{tabular}
\end{adjustbox}
\caption*{\small $\dagger$ All empirical coverage estimators returned $1$ for $\hat{\Delta}_{MR}$}
\end{table}
As expected, $\hat{\Delta}_{OLS}$ is a confounded estimator of $\text{E}Y^{t=1} - \text{E}Y^{t=0} = \Delta =10$. However, this is not the case for $\hat{\Delta}_{MR}$, which demonstrates convergence behavior towards the true parameter value, albeit at a sub-optimal rate. Reiteration of a poignant fact here is valuable: $\hat{MR}$ demonstrates this behavior without adjustment. Moreover, although the sampling conditions supposed here were non-informative, this was unnecessary. Insofar as the supports are preserved, informative sampling conditions are unimportant w.r.t. statistical consistency, although they might impact convergence rates. These results are also consistent with prior discussions on the mid-range. We expect the mid-range to possess more favorable properties when non-negligible mass or density rests in the tails, which is the case when $\tau = 25$ for these simulations.

Next, we present a restricted discrete analogue to the last experiment. $C$ and $T$ retain their definitions, but $\epsilon \sim Binom(\tau, 2^{-1})$ for $\tau \in \{30, 50 \}$. Otherwise, $Y^{t=0} = 10C + e$, $Y^{t=1} = Y^{t=0} + 5$, and $Y = Y^{t=1}T + (1-T)Y^{t=0}$ as before. For this simulation, we employ the Hoeffing bootstrap for $\mathcal{U}$ statistics once more. \textbf{Table 5} presents the results.
\begin{table}[H]
\caption{Discrete Effect Estimators: $\Delta =5$}
\centering
\begin{adjustbox}{width=1\textwidth}
\begin{tabular}{ccc|ccc|ccc}
  \hline
   \\[-3.5\smallskipamount]
 & n & $\hat{\Delta}_{OLS}$ & $\hat{\Delta}_{\hat{Av}}$ & $\hat{CI}_H$ & $\hat{EP}_H$ & $\hat{\Delta}_{\hat{MR}}$ & $\hat{CI}_{H}$ & $EP_H$  \\ 
  \\[-3.5\smallskipamount]
  \hline
   \\[-3.5\smallskipamount]
$\tau = 30$ & 500 & 9.694 & 6.059 & (0.55, 11.57) & 0.681 & 5.899 & (-0.47, 12.27) & 0.447 \\ 
   & 2500 & 9.875 & 5.91 & (1.34, 10.48) & 0.871 & 5.861 & (0.63, 11.09) & 0.680 \\ 
   & 5000 & 10.017 & 5.8 & (1.45, 10.15) & 0.890 & 5.757 & (0.83, 10.69) & 0.728 \\ 
   & 10000 & 10.037 & 5.765 & (1.68, 9.85) & 0.921 & 5.729 & (1.19, 10.27) & 0.805 \\ 
     \\[-3.5\smallskipamount]
  \hline
     \\[-3.5\smallskipamount]
 $\tau = 50$  & 500 & 9.696 & 6.443 & (-0.16, 13.04) &  0.458  & 6.192 & (-2.11, 14.49) & 0.236 \\ 
   & 2500 & 9.875 & 6.268 & (0.75, 11.79) & 0.728 & 6.150 & (-0.69, 12.99) & 0.419 \\ 
   & 5000 & 10.017 & 6.147 & (0.87, 11.42) & 0.734 & 6.052 & (-0.48, 12.59) & 0.457 \\ 
   & 10000 & 10.036 & 6.074 & (1.08, 11.07) & 0.804 & 5.995 & (-0.03, 12.02) & 0.551 \\ 
     \\[-3.5\smallskipamount]
   \hline
\end{tabular}
\end{adjustbox}
\caption*{\small $\dagger$ All empirical coverage estimators were $> .998$}
\end{table}
The discrete estimators demonstrate finite sample bias. Nonetheless, they also demonstrate convergence towards $\Delta$ as $n$ becomes large, while $\hat{\Delta}_{OLS}$ remains confounded. Although $\Delta=5$ is a small to moderate effect size, $\hat{\Delta}_{\hat{Av}}$ shows ample power to detect it under the null hypotheses that $\Delta = 0$ for reasonable sample sizes when $\tau$ is set to $30$. Traditional levels of acceptable power are only achieved by the $\hat{\Delta}_{\hat{Av}}$ estimator at $n=10000$ when $\tau = 50$. The increase in variance and the lowered likelihood of observing some values of the support hinder both estimators' performance. Although the mid-range estimator seems to possess less bias for these simulations, the plug-in estimator seems to possess more power.

Our last experiment demonstrates how linear regression can still be used for causal inference when mean exchangeability does not hold but each error term is a $\mathcal{U}$ variable. To this end, we use a new setup for any measurable function $g$: $T \sim Bern(.3)$, $U_1 \sim TN(m=-10, M=10, \mu=0, \sigma =2)$, and $\mu_T = 100 + 20T$. Then we will say $U_2 \sim TN(m= \mu_T - g(T)-10, M= \mu_T - g(T) + 10, \mu = 0, \sigma = 2)$, $Y_T = \mu_T + U_1$, and $Y^T = g(T) + U_2$.

For simplicity, we set $g(T) = 90 + 10T$. This model structure can be amended to include more covariates. However, this is unnecessary for our demonstration. What is important is that, conditional on a design matrix $\mathbf{x}$, the functional averages are preserved and that linearity holds for $Y_{\mathbf{x}}$. The true generating process for the counterfactual distribution can be unknown. Here, $\Delta = 20$. However, confounding is present since $\text{E}Y^{t=1} - \text{E}Y^{t=0} = 10 \neq \text{E}Y_{t=1} - \text{E}Y_{t=0} = 20$.

The results are provided in \textbf{Table 6}. The $\hat{CI}_t$ column provides the arithmetic average of the standard $t$-distribution confidence set endpoints. For reference, we also include $\hat{CI}_{\mathcal{U}}$ for confidence sets constructed from the concentration inequality in \textbf{Table 1} for $\mathcal{U}$ errors.

$\hat{CI}_{\mathcal{U}}$ is constructed as follows. Say $\mathbf{x}$ is the design matrix for a regression to estimate $\boldsymbol{\beta}$ and $\mathbf{w} = (\mathbf{x}^{\top} \mathbf{x})^{-1} \mathbf{x}^{\top}$. Then $\boldsymbol{\hat{\beta}} - \boldsymbol{\beta} = \mathbf{w} \boldsymbol{\epsilon}$. Therefore, $\hat{\beta}_T - \beta_T = \sum_{i=1}^n w_{i, s} \epsilon_i$, where $s$ is the row of $\mathbf{w}$ corresponding to treatment feature. Algebraic rearrangement of the concentration inequality then yields confidence sets of the form $\hat{\beta}_{T} \pm \sqrt{\sum_{i=1}^n R_i^2} \cdot \sqrt{6^{-1} \text{log}(2/\alpha)} $, where $R_i$ is the population range of $w_{i, s} \epsilon_i$. Under the assumption of a valid mean model specification, the maximum of the supports of the $\epsilon_i$ is feasibly estimable with $\hat{e}_{(n)}$ WLOG, where $\hat{e}$ is a typical residual. Hence, we can use an approximate confidence set of the following form when the extremes of the support of $Y$ are unknown: $\hat{\beta}_{T} \pm \{ \hat{e}_{(n)} - \hat{e}_{(1)} \}  \cdot \sqrt{\sum_{i=1}^n w^2_{i, s}} \cdot \sqrt{6^{-1} \text{log}(2/\alpha)} $. We contrast these confidence sets to those constructed with the Hoeffding bootstrap of Section 2.
\begin{table}[H]
\centering
\caption{Linear Regression for Functional Averages}
\begin{tabular}{cccccc}
  \hline
   \\[-3.5\smallskipamount] 
 & n & $\hat{\Delta}_{OLS}$ & $\hat{CI}_t$ & $\hat{CI}_{\mathcal{U}}$ & $\hat{CI}_{H}$ \\ 
  \\[-3.5\smallskipamount] 
  \hline
   \\[-3.5\smallskipamount] 
 & 500 & 20 & (19.62, 20.37) & (19.09, 20.91) & (19.9, 21.58) \\ 
   & 2500 & --- & (19.83, 20.17) & (19.52, 20.48) & (19.28, 20.71) \\ 
   & 5000 & --- & (19.88, 20.12) & (19.64, 20.35) & (19.49, 20.51) \\ 
   & 10000 & --- & (19.91, 20.09) & (19.74, 20.26) & (19.64, 20.36) \\ 
    \\[-3.5\smallskipamount] 
   \hline
\end{tabular}
\end{table}
Since the properties of linear regression are well understood, a thorough discussion is unnecessary. The results again substantiate the utility of $\mathcal{U}$ random variables. Under their framework, efficient estimation of causal parameters is more readily achievable, especially if mutual independence is a feasible assumption.

\section{A Data Application}

In this section, we employ NHEFS data to demonstrate our concepts. The NHEFS conducted medical examinations from 1971-1975 from non-institutionalized civilian adults aged 24-74 (N = 14,407) in the United States as part of a national probability sample. Follow-up surveys were then administered in 1982, 1984, and subsequent years to collect measurements for behavioral, nutritional, and clinical variables. Further documentation is available elsewhere \cite{madans198610}. The subset of data we use here (n=1,479) originates from the original 1971 medical examination and follow-up in 1982.

Exercise (0: moderate to much; 1: little to none) is the treatment variable ($T$) of interest. Age, chronic bronchitis/emphysema diagnosis (1: yes; 0: never), education attained in 1971 (1: $<$  8th grade; 2: HS dropout; 3: HS; 4: college dropout; 5: college), income, race (1: non-white; 0: white), sex (1: female; 0: male), years smoking, alcohol frequency, and weight (kilograms) are utilized as adjusting covariates (henceforth denoted as $\mathbf{L}$). Although SBP was measured with integer values, it is still treated as continuous in most instances. Pertinently, $T$ and most covariates were all measured in 1971. Only SBP and weight were measured in 1982. All continuous covariates are centered on their observed sample means for this analysis.

Here, we are interested in seeing if exercise exerted a causal effect on SBP in smokers. We aim to estimate $\text{E} \{ \text{Av}(Y^{t=1}_{\mathbf{L}}) - \text{Av}(Y^{t=0}_{\mathbf{L}}) \}$ and $\text{E} \{ \text{Av}(Y^{t=1}_{S}) - \text{Av}(Y^{t=0}_{S}) \}$ as summary causal effects, where $S=s$ represents a stratum that has been constructed from the quintiles of $\mathbf{\hat{e}(L)}$, the estimated propensity scores. Our goal is to estimate $\text{Av}(Y^{t=1}) - \text{Av}(Y^{t=0})$ and $\beta_t$ in addition.

To this end, race is treated as a confounder since it represents both genetic information and socio-historical constructions \cite{witzig1996medicalization}. In a similar vein, we choose to adjust for sex to account for possible biological influences, and since it is also an imperfect proxy for social institutions that can impact exercise habits, other health behaviors, and therefore blood pressure. All other variates mentioned are adjusted for since they are either known to effect both SBP and exercise habits directly or to act as conduits for more general institutional or ecological influences. Theoretically, adjusting for them can help to block backdoor paths from a subset of unknown confounders, which, although mostly irrelevant here in terms of their probabilistic effects, can still impact supports.
\newline
\paragraph{\textit{Methods}} The following methods are used to estimate possible causal effects: simple linear regression (LR), multiple linear regression with and without standardization (S; MR), linear regression adjusted for propensity score strata with standardization (PS), and sample functional average estimation (Av) with the discrete plug-in for eq. (2.1). Propensity scores are estimated with logistic regression using the same covariates as the regression model. No covariate transformations are used for the logit model, although we introduce higher order terms to the regression if it appears to improve linearity. Standard errors for all standardization estimators are calculated via a standard bootstrapping procedure with $B=1,000$ replications. We use the Hoeffding bootstrap of Section 3 with the same value of $B$ for the eq. (2.1) plug-in. Standard $ t$-statistic-based confidence sets are employed for the LR and MR models.

All tests are conducted at the $\alpha = .05$ level using R version 4.2.2 statistical software \cite{r}. The assumptions of $\mathcal{U}$ status and valid mean model specification are substantiated through the inspection of residual versus fitted plots and empirical CDF plots.
\newline
\paragraph{\textit{Results}}
The coefficient estimate for the simple linear regression is $\hat{\beta}_t = -3.87$ (95$\%$ CI: -5.821  -1.912), while the plug-in functional average estimate for $\text{Av}(Y^{t=1}) - \text{Av}(Y^{t=0})$ is $-3.81$ (95$\%$ CI: -24.17, 17.42). Additional model results are available in \textbf{Table 7}.
\begin{table}[H]
\centering
\caption{Summary of Model Results}
\begin{tabular}{cccc}
\hline
Parameter & Method & Estimate & $95 \%$ Confidence Interval  \\
\hline
$\hat{\beta}_t$ & MR & -.74 & $[-2.57,  1.1]$ \\
$\text{Av}(Y^{t=1}) - \text{Av}(Y^{t=0})$ & Av &  -3.81 & $[-24.17, 17.42]$ \\
$\text{E} \{ \text{Av}(Y^{t=1}_{\mathbf{L}}) - \text{Av}(Y^{t=0}_{\mathbf{L}}) \}$ & S & -.74 & $[ -2.56, 1.09]$ \\
$\text{E} \{ \text{Av}(Y^{t=1}_{\mathbf{S}}) - \text{Av}(Y^{t=0}_{\mathbf{S}}) \}$ & PS & -1.014 & $[-2.99, .97]$ \\
\hline
\end{tabular}
\end{table}
Initial fitting procedures for the multiple linear regression model showed mild departures from linearity. The addition of a quadratic term for age appeared to improve model fit. From this reformed model, the estimate for the expected functional average change in SBP is $\hat{\text{E}} \{ \text{Av}(Y^{t=1}_{\mathbf{L}}) - \text{Av}(Y^{t=0}_{\mathbf{L}}) \} = -0.74$ (95$\%$ CI: -2.56, 1.09). This estimate is approximately equivalent to the estimated (functional) average change in blood pressure for adult smokers who exercised, conditional on all other covariates: $-.74$ (95$\%$ CI: -2.57, 1.1). Finally, the estimate of the expected functional average change in SBP w.r.t. the propensity score stratified model is $\hat{\text{E}} \{  \text{Av}(Y^{t=1}_{S}) - \text{Av}(Y^{t=0}_{S})\} = -1.014$ (95$\%$ CI: -2.99, .97).
\newline
\paragraph*{\textit{Model checking}} Empirical CDF plots for the SBP distributions are presented by propensity score strata and by treatment status in \textbf{Figure 1}. The residual versus fitted plot for the MR model is also included.

\begin{figure}[H]
\caption{Model Validation Plots}
\includegraphics[scale=.6]{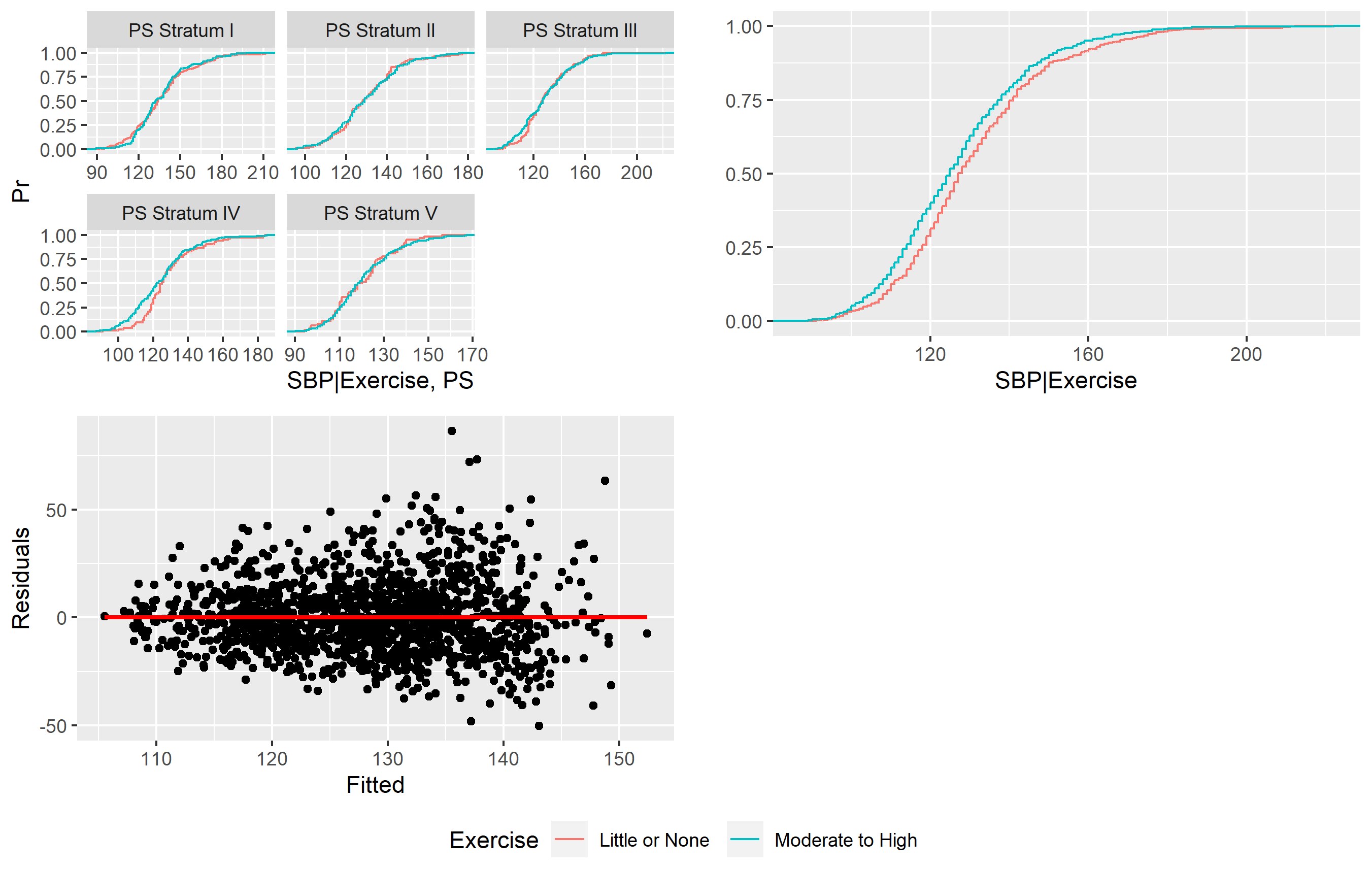} 
\end{figure}
We note no remarkable departures from linearity in the residual versus fitted plot for the multiple linear regression model. Moreover, the residuals appear to possess an approximately symmetric spread around the zero horizontal line, although some outlying points appear to violate the assumption that the errors are supported around symmetric extremes. Altogether, although departures from strict $\mathcal{U}$ status are observable, the assumption of $\mathcal{U}$ status appears to be feasibly met. The stratified empirical CDF plots in \textbf{Figure 1} do not appear to corroborate approximate sum-symmetric behavior since they show more area below the functional lines than above in most cases. Nevertheless, we note, conditional on an arbitrary stratum besides IV, that the area above each empirical CDF function appears roughly equal by exercise level. This alone is strong visual evidence that no difference in functional averages exists. The empirical CDF plots for the non-adjusted conditional distributions do not support the supposition of sum-symmetry. Again, this is because the areas above and below each function are not approximately equal. Hence, the results of the simple $t$-test cannot be afforded a causal interpretation w.r.t. a change in functional average.
\newline
\paragraph{\textit{Discussion}} Since the NHEFS was a national probability sample of non-institutionalized adults, there is little reason to believe that C4 was not fulfilled, conditional on our adjusting covariates. Recall that positing the opposing notion in this context is to affirm---for individuals who smoked---that there were possible values of SBP in each treatment population that had zero probability of being observed in the sample. Insofar as the NHEFS survey was truly a probability sample and hence non-informative, rejecting C4 also means that these potential values of SBP could never be observed in the real world. In conjunction with the fact that $\mathcal{U}$ status appeared to be approximately verified for the basic multiple regression model, we are confident that---conditional on our covariates---the expected (functional average) causal effect of exercise upon SBP in smokers is plausibly within a neighborhood of zero. This conclusion is further corroborated by the validity of C4 and the direct estimate of the difference in functional averages, which was also not statistically significantly different from zero. However, this test was hampered by the fact that each conditional SBP population possessed a relatively light right tail and the sample size was small.

Vitally, we have no reason to believe that we successfully adjusted for all confounding variables. Hence, we do not purport to interpret the former two effect estimates in terms of expected treatment effects. However, if mean exchangeability (and positivity) did hold, they would also be estimates of this contrast.

It is also apropos to note that the consistency assumption might be violated in this analysis. This is because respondents were asked if they exercised little to none, moderately, or much; however, the meanings of these words possess no absolutism. Hence, it is possible that multiple exercise treatments actually existed under the premise of one coding. This does not undermine what is formally specified in C4 conditional on the variable observed, although it does complicate the generalizability of the results if present.

\section{Conclusion}
In this article, we demonstrated that causal inference is achievable in the absence of mean exchangeability if the support of the counterfactual distribution is preserved. Moreover, we offered exposition on the possible utility and scientific meaningfulness of functional average change. To overcome some of the difficulties of functional average estimation, we introduced a simple class of random variables---the $\mathcal{U}$ class---that possesses a milieu of practical properties. Using the $\mathcal{U}$ random variable framework, we showed that ubiquitously employed statistical procedures produce estimates with causal interpretations under exceptionally mild conditions, many of which are already supposed in most applied settings to investigate associations. Hence, even if a researcher fails to control for all confounding variables, she still might be left with a second-prize of sorts, and one that possesses salient causal meaning. Since uncontrolled confounding is safely assumed to be nearly omnipresent outside of toy examples, we believe that this framework provides a strong defense of elementary methods. Further work is of course due. We observe that we did not pursue the sample minimum and maximum as counterfactual estimators in and of themselves, although the set of assumptions employed here also establish their utility for estimands with potential causal interpretations. Developing this area of theory will most certainly be advantageous.

Lastly, we also presented a new approach to the bootstrapping process. We called this approach the Hoeffding bootstrap. Although a less conservative form of it was proven, and strictly for the case s.t. the maximum bootstrap statistic is discrete or admits a density, we resorted to a defensible citation of the principle of indifference to extend a slightly more conservative version of it to a wider class of statistics. This is a good start. However, this cannot be the end. More theoretical work on the relationship between the extremes of the support of the empirical distribution of the statistic and those of the population distribution will most certainly provide fruit. A defensible set of sufficient conditions that ensure the approach more generally will do much to lighten the burden of uncertainty.
\newline
\paragraph*{\textbf{Acknowledgements}:} The NHEFS data was acquired from Dr. Miguel Hernan's faculty website ( https://www.hsph.harvard.edu/miguel-hernan/causal-inference-book/). We thank Dr. Hernan for making it accessible. Furthermore, we would like to thank the editor and reviewers for their diligent and insightful commentary, which improved the quality of this manuscript.
\newline
\paragraph*{\textbf{Funding information}:} We have no funding information to declare.
\newline\paragraph*{\textbf{Conflicts of Interest}:} Authors state no conflicts of interest. 

\bibliographystyle{vancouver}
\bibliography{ms}
\ifarXiv
    

\section*{Supplementary material}

# THE FUNCTIONAL AVERAGE TREATMENT EFFECT

## Appendix A. Supplementary material

### A.1. Convergence of Random Sums.

The manuscript supposes fairly general dependency conditions and reasons about random sums in the final sections. Hence, it is important to establish the statistical consistency of random sums under C5. That is accomplished in Proposition A.1. For this document, $I = \{1, 2, 3, \ldots, n\}$.

**Proposition A.1.** *Suppose C5 and let $S_n = \sum_{i=1}^{n} w_{n,i} Y_i$ for any set of constants $\{w_{n,i}\}_{i \in I}$ that are not all equal to zero s.t. $w_{n,i} = O(n^{-1})$ and say that $Var(Y_i) < \infty$ for $\forall i$. Then $S_n \xrightarrow{a.s.} ES_n$ as $n \to \infty$.*

*Proof.* The proof uses a variance identity established by Sparkes and Zhang [1]. For an arbitrary random sum $S_n$ under our premises, $\mathrm{Var}(S_n) = \{1 + \mu_n \phi_n\} \sum_{i=1}^{n} w_{n,i}^2 \mathrm{Var}(Y_i) \leq \{1 + \mu_n C\} \sum_{i=1}^{n} w_{n,i}^2 \mathrm{Var}(Y_i)$ for some positive constant $C$. Now, denote $Z_{k_*} = \sup_{k > n} |S_n - ES_n| = |S_{k_*} - ES_{k_*}|$. Let $\epsilon > 0$ be arbitrary and say $\max_{i \in I}\{\mathrm{Var}(Y_i)\} \leq M_*$ and $\max_{i \in I}(|w_{n,i}|) = w_{n,*}$. Then by Markov's inequality:

$$Pr(Z_{k_*} > \epsilon) \leq \epsilon^{-2}\{1 + \mu_{k_*} C\} \sum_{i=1}^{k_*} w_{k_*,i}^2 \mathrm{Var}(Y_i)$$

$$\leq \epsilon^{-2}\{1 + \mu_{k_*} C\} \cdot w_{k_*,*}^2 k_* M_*$$

Since $w_{k_*,*}^2 = O(k_*^{-2})$ and $k_*^{-1}\mu_{k_*} \to 0$ as $n \to \infty$ and hence $k_* \to \infty$:

$$\lim_{n \to \infty} Pr(Z_{k_*} > \epsilon) \leq \lim_{n \to \infty} \epsilon^{-2}\{1 + \mu_{k_*} C\} \cdot w_{k_*,*}^2 k_* M_* = 0$$

Hence, $Z_{k_*} \xrightarrow{p} 0$. However, this implies that $S_n \xrightarrow{a.s.} ES_n$. $\square$

### A.2. Hoeffding's Bootstrap.

In the paper we also introduced estimators for the functional average. We showed that these estimators are consistent

in the statistical sense under very general dependence conditions. Unfortunately, however, their properties are otherwise nebulous. This is unsatisfactory for inference. We made the claim that the bootstrap could be repurposed to (conservatively) solve this problem under some conditions. We now qualify this statement.

The literature on on the bootstrap is vast and we do not review it. Essentially, the bootstrap re-samples from the observed empirical distribution to emulate the population distribution of interest. Insofar as the empirical distribution function $\hat{F}_n(x)$ is strongly consistent, the targeted statistic is a smooth functional of $F(x)$, and the researcher's theory about the underlying dependency structure that informs the re-sampling process is approximately correct, it provides a feasible approach to inference. Unfortunately, our theories pertaining to the dependencies between observations are not empirically verifiable as a whole. Like all models, they are also probably far off the mark and invalid.

Our main approach does not condition on the observed sample for re-draws. Rather, it treats them as stochastic. To this end, say that $\hat{\theta} = T_0(X_1, X_2, \ldots, X_n) : \mathbb{R}^n \to \mathbb{R}$ is our statistic from sample $\zeta = \{X_i\}_{i \in I}$. We will assume a simple random sampling process with replacement (SRSWOR), as usual, for executing the re-sampling process. This assumption is not strictly necessary—once again, sampling can be directed by a theory of dependence—but we avoid this here. Denote a $k$th re-sampling $\zeta_k$ a collection of the newly sampled variables. For instance, $\zeta_k = \{X_1, X_2, X_2, X_4, \ldots. X_n\}$ if $X_2$ has been sampled twice and only $X_3$ has been replaced (here we are abusing set notation). Pertinently, we then note that $T_k(X_1, X_2, \ldots, X_n) : \mathbb{R}^{n-1} \to \mathbb{R}$, but $T_k$ still has the same 'working' magnitude as $T$ since $X_2$ has simply taken the place of $X_3$ and occupied both arguments of the function. For our example, this would look like the following if $T$ was the arithmetic mean:

$$T_0(X_1, \ldots, X_n) = \bar{X} = n^{-1} \sum_{i=1}^n X_i$$

$$T_k(X_1, X_2, \ldots, X_n) = n^{-1}2 \cdot X_2 + n^{-1} \sum_{i \in I \setminus \{2,3\}} X_i$$

Hence, any 'bootstrap' of $K = K(n)$ re-samples results in a sample of random functions $\{T_k\}_{k \in \mathcal{K}}$, where $\mathcal{K} = \{0, 1, 2, \ldots, K\}$ in this exploration. Note that we are making $K$ an implicit function of $n$. We do not specify this function. We only require that $K \to \infty$ when $n \to \infty$.

From here, we rely on the same mechanism that was used elsewhere to prove the statistical consistency of the sample maximum [1]. Note that for any two elements $s, r \in \mathcal{K}$, $\Pr(T_s \leq t | T_r \leq t) = 1$ if and only if $\Pr(T_s \leq t \cap T_r \leq t) = \Pr(T_r \leq t)$, which indicates that the event $\{\omega \in \Omega : T(\omega)_s \leq t\}$

is implied by the event $\{\omega \in \Omega : T(\omega)_r \leq t\}$. When this occurs—for an arbitrary assortment of conditioning events—the conditional CDF is called trivial. Altogether, note that there are $1 \leq Q(n) \leq K + 1$ non-trivial conditional CDFs. More detail is required when defining $Q(n)$, but we leave this for the next proposition.

For clarity, one possible situation that can give rise to a trivial CDF for the arithmetic mean is as follows. Say $X_2 = \beta X_1$ and observe $T_r(X_1, X_2) = n^{-1}\{c_1 X_1 + c_2 X_2\}$ s.t. $c_1 + c_2 = n$ and $T_s(X_1) = n^{-1}\sum_{i=1}^{n} X_1 = X_1$. These functions correspond to the event that $X_1$ is sampled $n$ consecutive times and $X_1, X_2$ are sampled $c_1$ and $c_2$ times. Since $X_2 = \beta X_1$, $T_r(X_1, X_2) = n^{-1}\{c_1 + c_2\beta\}X_1$ and $T_r$ and $T_s$ are monotonic in relation. Hence, $\Pr(T_r \leq t | T_s \leq t) \in \{0, 1\}$ and depends on the constant $\beta$.

Delineating the number of conditional CDFs that are trivially one for an arbitrary statistic is not possible. However, it should be apparent that a diverging number of the $n^n$ re-samples should produce non-trivial conditional CDFs for reasonable statistics, provided a dependency structure that is not extreme. By extreme, we mean those dependency structures s.t. $\zeta$ is in reality a sample of random variables that are monotonic transformations of one another, or one s.t. the random variables are closely related in a system of constraints. In practice, we would not expect this to be the case in most circumstances.

Next, construct $T_{(K)} = \max_{k \in \mathcal{K}}\{T_k\}$ and note that $\Pr(T_{(K)} \leq t) = \Pr(T_K \leq t, T_{K-1} \leq t, \ldots, T_0 \leq t)$, where $T_0$ is the statistic of the original sample. We can then do a typical factoring WLOG: $\Pr(T_K \leq t, T_{K-1} \leq t, \ldots, T_0 \leq t) = \Pr(T_K \leq t | T_{K-1} \leq t, T_{K-2} \leq t, \ldots, T_0 \leq t) \cdot \Pr(T_{K-1} \leq t | T_{K-2} \leq t, T_{K-3} \leq t, \ldots, T_0 \leq t) \cdots \Pr(T_1 \leq t | T_0 \leq t) \cdot \Pr(T_0 \leq t)$. Recall: many of these conditional CDFs will evaluate to one. We call this sequence of conditional CDFs $\mathcal{C}$.

**Proposition A.2.** *Assume a sample of $n$ random variables $\zeta = \{X_i\}_{i \in I}$ and let $T : \mathbb{R}^n \to \mathbb{R}$ be a measurable function. Denote $\{T_k\}_{k \in \mathcal{K}}$ for $\mathcal{K} = \{0, 1, 2, \ldots, K\}$ as a set of random functions s.t. $T_k : \mathbb{R}^l \to \mathbb{R}$ for $l \in \mathbb{N} \leq n$ and $T_k$ is a function of the collection of random variables re-sampled from $\zeta$ as previously described. Observe that $T_0 = T(X_1, X_2, X_3, \ldots, X_n)$, the function of the original sample, is included. Next, suppose for an arbitrary $k \in \mathcal{K}$ that $\mathcal{S}_k$, the support of $T_k$, is a finite set of real numbers with maximum $M_k$. Furthermore, define $T_{(K)} = \max_{k \in \mathcal{K}}(T_k) \sim Pr(T_{(K)} \leq t)$ and suppose it is a discrete random variable or possesses a density when continuous. Finally, denote $1 \leq Q(n) \leq K + 1$ as the number of conditional CDFs in the sequence $\mathcal{C}$ that are strictly less than one when $t < \max_{k \in \mathcal{K}}(M_k) = M$. Then, if $Q(n) \to \infty$ as $n \to \infty$ and $K \to \infty$, $T_{(K)} \overset{a.s.}{\to} M$.*

*Proof.* Suppose the premises. The style of this proof also follows Sparkes and Zhang [1]. First, note that the support of $T_{(K)}$, say $\mathcal{S}$, is a subset of the union of the supports of each $T_k$ with a supremum that is the maximum of their respective suprema. Hence, $\mathcal{S} \subseteq \bigcup_{k \in \mathcal{K}} \mathcal{S}_k$ s.t. $M = \max_{k \in \mathcal{K}}(M_k) = \max(\mathcal{S})$. Now, observe for an arbitrary $t$ that $\Pr(T_{(K)} \leq t) = \Pr(T_K \leq t | T_{K-1} \leq t, T_{K-2} \leq t, \ldots, T_0 \leq t) \cdot \Pr(T_{K-1} \leq t | T_{K-2} \leq t, T_{K-3} \leq t, \ldots, T_0 \leq t) \cdots \Pr(T_1 \leq t | T_0 \leq t) \cdot \Pr(T_0 \leq t)$ and say $Q(n)$ is the number of conditional CDFs that are less than one in the sequence $C$ when $1_{t < M} = 1$. For brevity, say $\Pr(T_K \leq t | T_{K-1} \leq t, T_{K-2} \leq t, \ldots, T_0 \leq t) = \Pr(T_K \leq t | \mathcal{A}_K)$ WLOG under the convention that $\Pr(T_0 \leq t | \mathcal{A}_0) = \Pr(T_0 \leq t)$.

Moreover, denote $F(t)$ as the maximum of these conditional CDFs that are strictly less than unity when $t < M$. Then there exists an $\epsilon > 0$ s.t. $F(t) = 1 - \epsilon$ and:

$$
\begin{aligned}
\Pr(T_{(K)} \leq t) &= \Pr(T_K \leq t | \mathcal{A}_K) \cdot \Pr(T_{K-1} \leq t | \mathcal{A}_{K-1}) \cdots \Pr(T_0 \leq t) \\
&\leq 1^{K+1-Q(n)} F(t)^{Q(n)} \\
&= 1^{K+1-Q(n)} \cdot (1 - \epsilon)^{Q(n)} \\
&= (1 - \epsilon)^{Q(n)}
\end{aligned}
$$

Next, as a minor lemma, note that for any discrete random variable $Z$ with CDF $F(z)$ and mass function $f(z)$ on bounded support $\{m, z_1, z_2, \ldots, z_{R-2}, M\}$ with maximum and minimum $\{m, M\}$ respectively, $\int_m^M F(z) dz = M - \mathrm{E}Z$ since $F(z)$ is a step-function that can be integrated over $[m, M]$. To see this, express $\int_m^M F(z) dz = (z_1 - m) \cdot f(m) + (z_2 - z_1) \cdot \{f(m) + f(z_1)\} + (z_3 - z_2) \cdot \{f(m) + f(z_1) + f(z_2)\} + \cdots + (M - z_{R-2}) \cdot \{1 - f(M)\}$. Expand this to note the pattern of cancellations:

$$
\begin{aligned}
\int_m^M F(z) dz &= (z_1 - m) \cdot f(m) + (z_2 - z_1) \cdot \{f(m) + f(z_1)\} \\
&\qquad + (z_3 - z_2) \cdot \{f(m) + f(z_1) + f(z_2)\} + \cdots \\
&= -m f(m) + z_1 f(m) - z_1\{f(m) + f(z_1)\} + z_2\{f(m) + f(z_1)\} \\
&\qquad - z_2\{f(m) + f(z_1) + f(z_2)\} + \cdots \\
&= -m f(m) - \sum_{i=1}^{R-2} z_i f(z_i) + M - M f(M) \\
&= M - \mathrm{E}Z
\end{aligned}
$$

Therefore, $\mathrm{E}Z = M - \int_m^M F(z) dz$.

Furthermore, for any continuous random variable $Y$ with a density $f(y)$ and CDF $F(y)$ s.t. $\min(\mathcal{S}_Y) = m > -M$ and $\max(\mathcal{S}_Y) = M$, it also follows that $\mathrm{E}Y = \int_{-M}^M \{2^{-1} - F(y)\} dy = M - \int_m^M F(y) dy$.

Now, designate $m = \min(\mathcal{S})$ specifically. Then, for either of our cases, $\mathrm{E}T_{(K)} = M - \int_m^M \Pr(T_{(K)} \le t)dt$ and:

$$M - \{M - m\} \cdot (1 - \epsilon)^{Q(n)} \le \mathrm{E}T_{(K)} \le M$$

It then of course follows that:

$$M - \{M - m\} \cdot \lim_{n \to \infty}(1 - \epsilon)^{Q(n)} \le \lim_{n \to \infty} \mathrm{E}T_{(K)} \le M$$

Which of course implies that:

$$M \le \lim_{n \to \infty} \mathrm{E}T_{(K)} \le M$$

And therefore that $\mathrm{E}T_{(K)} \to M$ as $n \to \infty$ and $K \to \infty$. This in turn implies that $T_{(K)}$ converges in distribution to the constant $M$ since it is a boundary value of the support, which implies that $T_{(K)} \xrightarrow{p} M$ as $n$ and $K$ become arbitrarily large.

Next, note that $\{T_{(K)}\}_{k \in \mathcal{K}}$ is a bounded and non-decreasing sequence of random variables. Hence, the point-wise limit of $T_{(K)}$ as $n \to \infty$ and $K \to \infty$ exists. Call this limit $L$. However, since $T_{(K)}$ converges in probability to $M$ as both $n \to \infty$ and $K \to \infty$, there is a sub-sequence of $\{T_{(K)}\}_{k \in \mathcal{K}}$ that converges almost surely to $M$. This then suggests that $T_{(K)} \xrightarrow{a.s.} M$ as $n \to \infty$ and $K \to \infty$ by the uniqueness of limits, which suggests $L = M$ almost surely. $\qquad\square$

As stated in the paper, the supposition that the density of the maximum statistic exists is non-trivial. This assumption, however, is commonplace and can be circumvented with the more conservative version justified by the principle of indifference. Also, since $\min_{k \in \mathcal{K}}(T_k) = -\max_{k \in \mathcal{K}}(-T_k)$, we can employ the last proposition to establish the statistical consistency of the minimum as well.

Now, we can make the utility of these propositions clear. First, consider the setup for Proposition A.2. For this case, since $T_0$ is included in $\{T_k\}_{k \in \mathcal{K}}$, $\max(\mathcal{S}_{T_0}) \le M$ and $\min(\mathcal{S}_{T_0}) \ge m$. Therefore, $M_0 - m_0 \le M - m$.

Under the conditions of this proposition, we know that we can estimate $M - m$ consistently. Since $T_0$ is a bounded random variable and therefore sub-Gaussian, a straightforward application of Hoeffding's inequality for an arbitrary $\epsilon > 0$ yields that:

$$\Pr(|T_0 - \mathrm{E}T_0| > \epsilon) \le 2\exp\{-(M_0 - m_0)^{-2} \cdot 2\epsilon^2\}$$

And this implies that:

$$\Pr(|T_0 - \mathrm{E}T_0| > \epsilon) \le 2\exp\{-(M - m)^{-2} \cdot 2\epsilon^2\}$$

One can then use the plug-in estimate from the bootstrap to produce $1 - \alpha$ confidence sets of the (informal) form $T_0 \pm (\hat{M} - \hat{m})\sqrt{2^{-1}\log(2/\alpha)}$, where $\hat{M}$

is the maximum of the bootstrapped statistics WLOG. Under the assumptions of the manuscript, and perhaps a more controlled re-sampling process s.t. each $T_k$ converges almost surely or in probability to its expectation with probability one for the first setup, we would expect $\hat{M} - \hat{m}$ to tend towards zero.

If $T_0 \in \mathcal{U}$, we can improve the confidence set to $T_0 \pm (\hat{M} - \hat{m}) \sqrt{6^{-1}\log(2/\alpha)}$. Recall that we expect $T_0$ to approach $\mathcal{U}$ status if its converges almost surely or in probability, or if it is in a scale family s.t. its density is proportional to its standard deviation, which tends to zero. This provides justification for the sharper confidence set, at least for moderately sized samples.

We are left with one final problem to consider. Often, the researcher is interested in constructing an approximate $1 - \alpha$ confidence set w.r.t. a parameter of interest, say $\theta$. However, while $\mathrm{E}T_0 \to \theta$ as $n \to \infty$, it is often the case that $\mathrm{E}T_0 \neq \theta$ for finite samples. Consequently, the $1 - \alpha$ confidence sets, which are built around $\mathrm{E}T_0$, can fall short of the nominal value. This is ultimately another reason to use a Hoeffding style bootstrap, however. We demonstrate this informally.

Trivially, we know that $T_0 \pm \{M - m\} \cdot \sqrt{2^{-1}\log(2/\alpha)}$ contains $\mathrm{E}T_0$ with probability one for an arbitrary $n$ when $0 < \alpha \leq 2\exp\{-2\}$. We designate $\eta(n) = \{M - m\} \cdot \sqrt{2^{-1}\log(2/\alpha)} > 0$. Next, suppose $S(n)$ is some function of $n$ s.t. $\sqrt{n}S(n)$ stabilizes the variance of $T_0$ and therefore that $\sqrt{n}S(n)\{T_0 - \theta\}$ converges almost surely to some random variable $D$ with expectation zero and finite variance. It is then implied that $\sqrt{n}S(n)\eta(n) \to C_* > 0$ as $n \to \infty$ or $\sqrt{n}S(n)\eta(n) \to \infty$.

Now, observe that if $\sqrt{n}S(n)\{T_0 - \theta\} \overset{a.s.}{\to} D$, then there is some random variable $Z_n$ s.t. $\sqrt{n}S(n)\{T_0 - \theta\} = D + Z_n$ and $Z_n \overset{a.s.}{\to} 0$. Then we can also say that $\sqrt{n}S(n)\{\mathrm{E}T_0 - \theta\} = \mathrm{E}Z_n$ and that $\mathrm{E}Z_n \to 0$ as $n \to \infty$. Since $\mathrm{E}Z_n \to 0$, but $\sqrt{n}S(n)\eta(n) \to C_* > 0$ or is positively infinite, this implies that $\mathrm{E}T_0 - \theta = \{\sqrt{n}S(n)\}^{-1}\mathrm{E}Z_n \to 0$ as $n \to \infty$ at a faster rate than $\eta(n)$. Now, let $N \in \mathbb{N}$ be the natural number s.t. for all $n > N$, $|\mathrm{E}T_0 - \theta| = 0 \leq \eta(n)$ and consider the almost sure event that $|T_0 - \mathrm{E}T_0| \leq \{M - m\} \cdot \sqrt{2^{-1}\log(2/\alpha)}$. Then:

$$T_0 - \{M - m\} \cdot \sqrt{2^{-1}\log(2/\alpha)} \leq \mathrm{E}T_0 \leq T_0 + \{M - m\} \cdot \sqrt{2^{-1}\log(2/\alpha)}$$

Is equivalent to:

$$T_0 - \eta(n) \leq \mathrm{E}T_0 - \theta + \theta \leq T_0 + \eta(n)$$

And:

$$T_0 - \eta(n) \leq \theta \leq T_0 + \eta(n)$$

In conclusion, then, at least for sufficiently large $n$, constructing a confidence set that contains $\mathrm{E}T_0$ with probability one of the form $T_0 \pm \{M -$

$m\} \cdot \sqrt{2^{-1}\log(2/\alpha)}$ around $\mathrm{E}T_0$ is equivalent to forming an almost sure confidence set around $\theta$.

We of course point out that this same logic applies without the added factor of $\sqrt{2^{-1}\log(2/\alpha)}$, and that more efficient bounds on $|T_0 - \mathrm{E}T_0|$ can be chosen. Since more efficient bounds come at the cost of constraints on the distribution of $T_0$, this was avoided. On the former matter, it is apropos to recognize that $\hat{M} - \hat{m}$ will be biased downward w.r.t. the population extremes of the bootstrap distribution for real life applications that make use of finite $n$ and $K$. Worse, its rate of convergence will be unknown and, in all likelihood, sub-optimal. Hence, the employment of the added factor in concordance with Hoeffding's inequality is intended to compensate for these deficiencies. Although employed somewhat arbitrarily, Hoeffding's inequality is well-established and provides an intuitive choice for penalization.

Nevertheless, since the confidence set constructed provides almost certain coverage for $\theta$ in theory, the under-estimation of the bootstrap range is not likely to undermine the cogency of confidence sets with at least a working $1 - \alpha$ coverage in a majority of applied circumstances. This is also true since $M_0 \le M$ WLOG.

As a final point, we recognize that this strategy will be very conservative in most circumstances. However, we do not see this as a fault. Excluding those situations s.t. it is riskier to fail to reject a null hypothesis, say, this methodology will always be less subject to doubt and hence will also supply more cogent scientific arguments. A researcher can always make use of constraints to derive a more efficient method for inference, but the results of this methodology will always be subject to more doubt as a consequence of the additional assumptions, which are false in all likelihood. The Hoeffding bootstrap—in the minimum—can serve as a tool for sensitivity analysis, e.g., for distinguishing what propositions are the most inscrutable. Moreover, when its premises are satisfied, the Hoeffding bootstrap (at least nominally) eliminates the multiple testing issue for sufficiently large $n$. This trivially follows from the fact that, if the probability of a Type I error is zero for an arbitrary test, the family-wise error rate is preserved by implication.

A.3. **Simulations.** We offer three simulation experiments as a proof of concept under some dependence conditions. The bootstrap is known to fail when estimating the minimum of $n$ independent uniform distributions. Hence, we start with this example. It is also known to fail for non-smooth functions. Hence, our second example is $T_0 = |\bar{Y}_n|$ when $\mathrm{E}\bar{Y}_n = 0$. The third is for the arithmetic mean.

We demonstrate robustness to dependence by incorporating scenarios of simple but relatively extreme dependence into all experiments. We call

these setups the 'sneaky twin,' 'sneaky decuplet,' and 'sneaky venti-cuplet' scenarios, which are used for the first, second, and third experiments respectively. Essentially, we draw $n/2$ independent $Y_i$ WLOG for the first case. However, for each $Y_i$ drawn, we include it in the sample a second time. The other two setups function analogously, but replace 2 with 10 and 20. This does not validate this type of bootstrap for all scenarios. However, it is extreme enough to demonstrate the method's utility since the sneaky setup is more extreme than a number of setups that are commonly supposed. For instance, if $T_0$ is a random sum, the sneaky venti-cuplet scenario invisibly increases its variance by a factor of twenty.

The basic bootstrapping procedures will only use $B = 500$ bootstrap samples. The only two exception will be for the sneaky twin and venti-cuplet scenarios used for the sample minimums and the means: here, we use $B = 2000$ and $B = 1500$ bootstrap samples respectively to compensate for the behavior of the sample minimum and more extreme dependence picture. All setups will enact simple random sampling with replacement without any theory of dependence. This is not preferable in practice. In applied settings, it is best to enact a sampling scheme that approximates what is known about the dependency structure. Here, an 'out of the box' approach is used to demonstrate robustness. A non-targeted sampling will produce more trivial samples on average when relatively extreme dependence is present. Thus, an adequate performance when blind is a good baseline.

A.3.1. *Uniform Experiment.* Set $T_k = \min_{j \in \zeta_k}(Y_j)$. For the case of mutual independence, note that an arbitrary $\Pr(T_k | \mathcal{A}_k)$ will equal one only if the $k$th re-sample is contained in any of the other bootstrapped samples. Out of $n$ original random variables, there are $2^n - 1$ unique $T_k$ possible. As $n \to \infty$ and $B \to \infty$, our conditions will be fulfilled.

Since it is not the focus of this manuscript, we did not derive the exact properties of the sneaky twin scenario. An informal treatment will suffice. Let $T_{0,*}$ be the sample minimum of the $n/2$ independent $Y_i$. Then it is apparent that $T_0 = T_{0,*}$, the minimum of the $n/2$ independent outcome variables. In this scenario, a simple random sample with replacement of the $n$ outcomes will result in a large number of trivial conditional CDFs. This is because there are at most $n/2$ independent variables to draw from and thus a large proportion of the re-samples will have strictly less than $n/2$ of the independent variables. This will lead to the evaluation of conditional CDFs of the form $\Pr(T_j \leq t | T_s \leq t, \ldots)$ where $T_j$ and $T_s$ are sample minimums composed of nested or *almost* nested samples. For such a setup, $B = 2000$

is low since convergence of the bootstrap distribution will be slow. Nevertheless, it is computationally feasible and in the neighborhood of a common choice for an 'approximate' bootstrap.

For reference, we also use employ the $\mathcal{U}$ Hoeffding bootstrap of Section 3. The results of this simulation are in **Table A1**. $\hat{C}I_H$ provides the average results for the Hoeffding bootstrap of Section 2 WLOG, while $\hat{C}I_{H_U}$ provides the results for the method of Section 3. These are conservative one-sided $1 - \alpha$ confidence sets with the average sample minimum provided for the right endpoint.

TABLE A1. Hoeffding for Sample Minimum

|  | n | $\hat{C}I_H$ | $EC_H$ | $\hat{C}I_{H_U}$ | $EC_{H_U}$ |
|---|---|---|---|---|---|
| Independent | 500 | (-0.013, 0.002] | 0.992 | (-0.028, 0.002] | 1 |
|  | 2500 | (-0.003, < .001] | 0.990 | (-0.006, < .001] | 1 |
|  | 5000 | (-0.001, < .001] | 0.987 | (-0.003, < .001] | 0.999 |
| Sneaky Twin | 500 | (-0.014, 0.004] | 0.930 | (-0.032, 0.004] | 0.984 |
|  | 2500 | (-0.003, 0.001] | 0.930 | (-0.006, 0.001] | 0.986 |
|  | 5000 | (-0.001, < .001] | 0.943 | (-0.003, < .001] | 0.987 |
|  | 10000 | (-0.001, < .001] | 0.945 | (-0.002, < .001] | 0.984 |

It is apparent that the coverage properties suffered for the sneaky twin scenario. However, this is partially due to the low number of bootstraps. Otherwise, its performance could have been improved upon if the re-draws were focused—and even if imperfectly—on the $n/2$ independent originals. Nonetheless, the Hoeffding bootstrap performed largely as intended for both methods. Note that this seems trivial since we already know that the parameter is zero. In actual practice, we would not know this.

From here, only the Hoeffding bootstrap of Section 2 is simulated. This is sufficient since the one from Section 3 will always perform more conservatively. Therefore, if the Section 2 method offers decent performance, so does the one from Section 3.

A.3.2. *Non-smooth Experiment.* For this experiment, each $Y_i \sim TN(m = -20, M = 20, \mu = 0, \sigma = 5)$. Again, we make use of the sneaky decuplet method for the non-independence cases. This is accomplished by drawing $n/10$ independent variables and then sneaking in an additional nine copies of each. The results are below.

TABLE A2.  Hoeffding for $|\bar{Y}|$

|  | n | $T_0$ | $\hat{C}I_H$ | $EC_H$ |
|---|---|---|---|---|
| Independent | 500 | 0.179 | (-0.987,1.344) | 1 |
|  | 2500 | 0.081 | (-0.443,0.605) | — |
|  | 5000 | 0.057 | (-0.313,0.428) | — |
| Sneaky decuplet | 500 | 0.560 | (-0.929,2.049) | 0.985 |
|  | 2500 | 0.253 | (-0.42, 0.927) | 0.989 |
|  | 5000 | 0.179 | (-0.298, 0.657) | 0.986 |
|  | 10000 | 0.122 | (-0.214, 0.458) | 0.990 |

The bootstrap behaves as intended.  Although this bootstrap does not capture the true distribution of $|\bar{Y}|$, this is not the goal.  Instead, it was to construct cogent confidence sets for a parameter on the boundary of a non-smooth function. This is accomplished.

A.3.3.  *Arithmetic Mean Experiment.*  The same outcome distribution as the non-smooth experiment is used here as well.  Since this experiment possesses a relatively extreme dependency structure, we provide results for $B = 500$ and $B = 1500$ bootstrap procedures for the venti-cuplet simulations.  Moreover, since this is a case where the standard bootstrap could be applied, we also provide a complementary simulation using the same random seed.  The results for the Hoeffding bootstrap are in **Table A3**, while the results for the results for the traditional bootstrap that uses a bootstrap variance estimate for a normal approximation is in **Table A4**.

TABLE A3.  Hoeffding for $\bar{Y}$

|  | n | $T_0$ | $\hat{C}I_H$ | $EC_H$ |
|---|---|---|---|---|
| Independent | 500 | -0.006 | (-1.845, 1.832) | 1 |
|  | 2500 | 0.001 | (-0.822, 0.825) | — |
|  | 5000 | -0.000 | (-0.583, 0.582) | — |
| Sneaky venti-cuplet | 500 | -0.017 | (-1.815, 1.781) | 0.902 |
| ($B = 500$) | 2500 | -0.015 | (-0.83, 0.801) | 0.927 |
|  | 5000 | -0.000 | (-0.585, 0.584) | 0.933 |
|  | 10000 | 0.005 | (-0.406, 0.416) | 0.934 |
| Sneaky venti-cuplet | 500 | 0.027 | (-1.958, 2.011) | 0.929 |
| ($B = 1500$) | 2500 | -0.007 | (-0.913, 0.899) | 0.963 |
|  | 5000 | 0.004 | (-0.64, 0.648) | 0.958 |
|  | 10000 | 0.009 | (-0.445, 0.463) | 0.953 |

TABLE A4. Bootstrap Normal Approximation

|  | n | $T_0$ | $\hat{C}I_t$ | $EC_t$ |
|---|---|---|---|---|
| Independent | 500 | -0.006 | (-0.443, 0.431) | 0.955 |
|  | 2500 | 0.001 | (-0.194, 0.197) | 0.95 |
|  | 5000 | -0.000 | (-0.139, 0.138) | 0.95 |
| Sneaky venti-cuplet | 500 | 0.027 | (-0.4, 0.454) | 0.318 |
| $B = 1500$ | 2500 | -0.007 | (-0.202, 0.188) | 0.343 |
|  | 5000 | 0.004 | (-0.134, 0.143) | 0.340 |
|  | 10000 | 0.009 | (-0.089, 0.107) | 0.328 |

As expected, a larger number of bootstrap samples improves the performance of the Hoeffding bootstrap. In contrast, the standard method fails.


\fi

\end{document}